\documentclass[11pt]{article}

\usepackage[T1]{fontenc}
\usepackage[latin9]{inputenc}
\usepackage{geometry}
\usepackage{amsmath}
\usepackage{amsthm}
\usepackage{amssymb}
\usepackage{color}
\usepackage{paralist}
\usepackage[unicode=true,pdfusetitle,
 bookmarks=true,bookmarksnumbered=false,bookmarksopen=false,
 breaklinks=false,pdfborder={0 0 0},pdfborderstyle={},backref=false,colorlinks=false]
 {hyperref}

\makeatletter

\usepackage[nameinlink,capitalise,noabbrev]{cleveref}

\hypersetup{%
    bookmarksnumbered, bookmarksopen=true, bookmarksopenlevel=1,%
}

\theoremstyle{plain}
\newtheorem{thm}{Theorem}[section]
\crefname{thm}{Theorem}{Theorems}
\theoremstyle{plain}
\newtheorem{lem}[thm]{Lemma}
\crefname{lem}{Lemma}{Lemmas}
\theoremstyle{plain}
\newtheorem{cor}[thm]{Corollary}
\theoremstyle{plain}
\newtheorem*{claim*}{Claim}
\crefname{claim}{Claim}{Claims}
\theoremstyle{definition}

\theoremstyle{plain}
\newtheorem{conjecture}{Conjecture}
\crefname{myconjecture}{Conjecture}{Conjectures}
\theoremstyle{plconjectureain}
\newtheorem{prop}[thm]{Proposition}
\theoremstyle{definition}

\theoremstyle{definition}

\theoremstyle{plain}
\newtheorem{claim}[thm]{Claim}
\newtheorem{fact}[thm]{Fact}
\crefformat{equation}{#2(#1)#3}
\crefname{appsec}{Appendix}{Appendices}

\date{}

\usepackage{appendix}

\crefformat{equation}{#2(#1)#3}

\let\originalleft\left
\let\originalright\right
\renewcommand{\left}{\mathopen{}\mathclose\bgroup\originalleft}
\renewcommand{\right}{\aftergroup\egroup\originalright}
\usepackage{pgfplots}
\usetikzlibrary{pgfplots.groupplots}
\usepackage{verbatim}

\makeatletter
\renewcommand*{\UrlTildeSpecial}{%
  \do\~{%
    \mbox{%
      \fontfamily{ptm}\selectfont
      \textasciitilde
    }%
  }%
}%
\let\Url@force@Tilde\UrlTildeSpecial
\makeatother

\makeatother

\begin{document}

\title{Anticoncentration for subgraph statistics}

\author{Matthew Kwan \thanks{Department of Mathematics, Stanford University, Stanford, CA 94305. Email: \href{mailto:mattkwan@stanford.edu} {\nolinkurl{mattkwan@stanford.edu}}. This research was done while the author was working at ETH Zurich, and is supported in part by SNSF project 178493.}\and
Benny Sudakov\thanks{Department of Mathematics, ETH, 8092 Z\"urich, Switzerland. Email:
\href{mailto:benjamin.sudakov@math.ethz.ch} {\nolinkurl{benjamin.sudakov@math.ethz.ch}}.
Research supported in part by SNSF grant 200021-175573.}\and Tuan Tran\thanks{Department of Mathematics, ETH, 8092 Z\"urich, Switzerland. Email:
\href{mailto:manh.tran@math.ethz.ch} {\nolinkurl{manh.tran@math.ethz.ch}}.
Research supported by the Humboldt Research Foundation.}}

\maketitle
\global\long\def\RR{\mathbb{R}}
\global\long\def\E{\mathbb{E}}
\global\long\def\Var{\operatorname{Var}}
\global\long\def\NN{\mathbb{N}}
\global\long\def\ZZ{\mathbb{Z}}
\global\long\def\one{\boldsymbol{1}}
\global\long\def\range#1{\left[#1\right]}
\global\long\def\Bin{\operatorname{Bin}}
\global\long\def\Ber{\operatorname{Ber}}
\global\long\def\Rad{\operatorname{Rad}}
\global\long\def\F{\mathcal{F}}
\global\long\def\S{\mathcal{S}}
\global\long\def\R{\mathcal{R}}
\global\long\def\B{\mathcal{B}}
\global\long\def\Q{\mathcal{Q}}
\global\long\def\D{\mathcal{D}}
\let\polishL\L
\DeclareRobustCommand{\L}{\ifmmode{\mathcal{L}}\else\polishL\fi}
\global\long\def\ind{\operatorname{ind}}
\global\long\def\x{\boldsymbol{\xi}}
\global\long\def\hyp{\operatorname{BL}}
\global\long\def\AS{\operatorname{AS}}
\global\long\def\Inf{\operatorname{Inf}}
\global\long\def\g{\boldsymbol{\gamma}}
\global\long\def\b{b}
\global\long\def\bb{\boldsymbol{b}}
\global\long\def\minell{\ell^{*}}
\global\long\def\st{\operatorname{St}}

\begin{abstract}
Consider integers $k,\ell$ such that $0\le \ell \le \binom{k}2$. Given a large graph $G$, what is the fraction of $k$-vertex subsets of $G$ which span exactly $\ell$ edges? When $G$ is empty or complete, and $\ell$ is zero or $\binom k 2$, this fraction can be exactly 1. On the other hand, if $\ell$ is far from these extreme values, one might expect that this fraction is substantially smaller than 1. This was recently proved by Alon, Hefetz, Krivelevich and Tyomkyn who intiated
the systematic study of this question and proposed several natural
conjectures.

Let $\minell=\min\{\ell,\binom{k}{2}-\ell\}$. Our main result is
that for any $k$ and $\ell$, the fraction of $k$-vertex subsets that span $\ell$ edges is at most $\log^{O\left(1\right)}\left(\minell/k\right)\sqrt{k/\minell}$, which is best-possible
up to the logarithmic factor. This improves on multiple results of
Alon, Hefetz, Krivelevich and Tyomkyn, and resolves one of their conjectures.
In addition, we also make some first steps towards some analogous
questions for hypergraphs.

Our proofs involve some Ramsey-type arguments, and a number of different
probabilistic tools, such as polynomial anticoncentration inequalities,
hypercontractivity, and a coupling trick for random variables defined
on a ``slice'' of the Boolean hypercube.
\end{abstract}

\section{Introduction}

For an $n$-vertex graph $G$ and some $0\le k\le n$, consider a
uniformly random set of $k$ vertices $A\subseteq V\left(G\right)$
and define the random variable $X_{G,k}=e\left(G\left[A\right]\right)$
to be the number of edges induced by the random $k$-set $A$. The
point probability $\Pr\left(X_{G,k}=\ell\right)$ is then the fraction
of $k$-vertex subsets of $G$ which induce exactly $\ell$ edges.
If $G$ is an empty graph and $\ell=0$, or if $G$ is a complete
graph and $\ell=\binom{k}{2}$, this probability is exactly one. However,
if $\ell$ is far from these extreme values, and $G$ is sufficiently
large, one might expect $\Pr\left(X_{G,k}=\ell\right)$ to be small.
For example, Ramsey's theorem tells us that all sufficiently large
graphs must have induced $k$-vertex subgraphs that are empty or complete,
so if $\ell\ne\{0,\binom{k}{2}\}$ and $G$ is sufficiently large
then certainly $\Pr\left(X_{G,k}=\ell\right)<1$. In general, what
upper bounds can we give on $\Pr\left(X_{G,k}=\ell\right)$ for large
$G$?

Recently, Alon, Hefetz, Krivelevich and Tyomkyn \cite{AHKT} initiated
the systematic study of this question, motivated by its connections to graph inducibility\footnote{Roughly speaking, the \emph{inducibility} of a graph $H$ measures the maximum number of induced copies of $H$ a large graph can have. This notion was introduced in 1975 by Pippenger and Golumbic \cite{PG75},
and has enjoyed a recent surge of interest; see for example \cite{BHLP16,HT,Yus,KNV}.}. They proved some upper bounds on $\Pr\left(X_{G,k}=\ell\right)$
for various values of $k$ and $\ell$, and made some appealing conjectures. To state these, we recall some of their notation. Let $I\left(n,k,\ell\right)=\max\left\{ \Pr\left(X_{G,k}=\ell\right):\left|V\left(G\right)\right|=n\right\} $
be the maximum value of $\Pr\left(X_{G,k}=\ell\right)$ over all $n$-vertex
graphs, and let $\ind\left(k,\ell\right)=\lim_{n\to\infty}I\left(n,k,\ell\right)$
(one can use a standard averaging argument to show that $I\left(n,k,\ell\right)$
is a monotone nonincreasing function of $n$). The following three
conjectures appear as \cite[Conjecture~1.1, Conjecture~6.1 and Conjecture~6.2]{AHKT}.
\begin{conjecture}
\label{conj:1/e}For all $0<\ell<\binom{k}{2}$ we have $\ind\left(k,\ell\right)\le1/e+o_{k}\left(1\right)$.
\end{conjecture}

\begin{conjecture}
\label{conj:superlinear}For all $k,\ell$ satisfying $\min\left\{ \ell,\binom{k}{2}-\ell\right\} =\omega_{k}\left(k\right)$,
we have $\ind\left(k,\ell\right)=o_{k}\left(1\right)$.
\end{conjecture}

\begin{conjecture}
\label{conj:sqrt}For all $k,\ell$ satisfying $\min\left\{ \ell,\binom{k}{2}-\ell\right\} =\Omega_{k}\left(k^{2}\right)$,
we have $\ind\left(k,\ell\right)=O\left(k^{-1/2}\right)$.
\end{conjecture}

The authors of \cite{AHKT} proved some partial results for all of
these conjectures. Specifically, under the assumptions of \cref{conj:1/e}
they proved that $\ind\left(k,\ell\right)=1-\Omega_{k}\left(1\right)$,
under the assumptions of \cref{conj:superlinear} they proved that
$\ind\left(k,\ell\right)\le1/2+o_{k}\left(1\right)$, and under the
assumptions of \cref{conj:sqrt} they proved that $\ind\left(k,\ell\right)=O\left(k^{-0.1}\right)$.

Our main result is the following theorem, simultaneously implying
\cref{conj:superlinear} and an asymptotic version of \cref{conj:sqrt}.
This improves two of the aforementioned results in \cite{AHKT}.
\begin{thm}
\label{thm:almost}For all $k$ and all $0\le\ell\le\binom{k}{2}$,
let $\minell=\min\left\{ \ell,\binom{k}{2}-\ell\right\} $. We have
\[
\ind\left(k,\ell\right)\le\log^{O\left(1\right)}\left(\minell/k\right)\sqrt{\frac{k}{\minell}}.
\]
\end{thm}
We remark that we allow the ``$O(1)$'' term to equal zero, so the above statement still makes sense (and is in fact trivial) if $\ell^*\le k$. Note that up to the logarithmic factor \cref{thm:almost} is essentially best-possible.
Indeed, for any $s\le k$ and any $f=\omega\left(1\right)$, let $n=fk$ and
consider the $n$-vertex complete bipartite graph $G=K_{fs,fk-fs}$.
Then for $\ell=s\left(k-s\right)$ we have
\[
\Pr\left(X_{G,k}=\ell\right)\ge\frac{\binom{fs}{s}\binom{fk-fs}{k-s}}{\binom{fk}{k}}=\Theta\left(\sqrt{\frac{k}{s\left(k-s\right)}}\right)=\Theta\left(\sqrt{\frac{k}{\ell}}\right).
\]
We prove \cref{thm:almost} in \cref{sec:almost}. Our proof depends
on a polynomial anticoncentration inequality due to Meka, Nguyen and
Vu \cite{MNV16}, which itself depends on a weak version of the so-called
Gotsman--Linial conjecture in the theory of Boolean functions,
proved by Kane \cite{Kan14}. Any improvements to this anticoncentration
inequality, potentially via progress towards the Gotsman--Linial
conjecture, would result in corresponding improvements to
\cref{thm:almost}. We discuss this further in \cref{sec:concluding}. 

It is also interesting to study related questions for hypergraphs;
indeed, in \cite{AHKT} the authors specifically suggested that a
natural analogue of \cref{conj:1/e} might also hold for $r$-uniform
hypergraphs. We make a first step in this direction, generalising
a result in \cite{AHKT}. For $0<\ell<\binom{k}{r}$ and an $r$-uniform
hypergraph $G$ with at least $k$ vertices, we may define $X_{G,k}$,
$I_{r}\left(n,k,\ell\right)$ and $\ind_{r}\left(k,\ell\right)$ in
the obvious way: $X_{G,k}$ is the number of edges induced by a uniformly
random $k$-vertex subset of $G$, $I_{r}\left(n,k,\ell\right)$ is
the maximum value of $\Pr\left(X_{G,k}=\ell\right)$ over $n$-vertex
$r$-uniform hypergraphs $G$, and $\ind_{r}\left(k,\ell\right)=\lim_{n\to\infty}I_{r}\left(n,k,\ell\right)$.
\begin{thm}
\label{thm:1-eps}For any $r$ there exists $\varepsilon=\varepsilon\left(r\right)>0$
such that for any $0<\ell<\binom{k}{r}$ we have $\ind_{r}\left(k,\ell\right)\le1-\varepsilon$.
\end{thm}

A proof of this theorem for graphs appears as \cite[Theorem~1.3]{AHKT},
and proceeds via a long and complicated fourth-moment calculation.
We give a 
short proof of \cref{thm:1-eps} in \cref{sec:1-eps}
using a hypercontractive inequality. For concreteness, we remark that
\cref{thm:1-eps} holds with $\varepsilon=2^{-4/3}3^{-16r}.$

Finally, the natural hypergraph generalisation of \cref{conj:sqrt}
is that for any fixed $r$ and any $k,\ell$ satisfying $\min\{\ell,\binom{k}{2}-\ell\}=\Omega_{k}\left(k^{r}\right)$,
we have $\ind_{r}\left(k,\ell\right)=O\left(k^{-1/2}\right)$. This
problem appears to be quite difficult; we make a first step in the
case $r=3$.
\begin{thm}
\label{thm:hypergraph-sqrt}For all $k,\ell$ satisfying $\minell=\min\left\{ \ell,\binom{k}{3}-\ell\right\} =\Omega_{k}\left(k^{3}\right)$,
we have
\[
\ind_{3}\left(k,\ell\right)\le\frac{\log^{O\left(1\right)}k}{\sqrt{k}}.
\]
\end{thm}

We prove \cref{thm:hypergraph-sqrt} in \cref{sec:hypergraph-sqrt}.

\subsection{\label{sec:outline}Discussion and main ideas of the proofs}

Let $A=\left(a_{xy}\right)_{x,y}$ be the adjacency matrix of a graph
$G$. We can express $X_{G,k}$ as a homogeneous quadratic polynomial
\[
\frac{1}{2}\x A\x^{T}=\sum_{1\le x<y\le n}a_{xy}\xi_{x}\xi_{y},
\]
where $\x=\left(\xi_{1},\dots,\xi_{n}\right)$ is a uniformly random
length-$n$ zero-one vector with exactly $k$ ones. To prove \cref{thm:almost}
we need to upper-bound $\ind\left(k,\ell\right)$, which essentially
comes down to upper-bounding probabilities of the form $\Pr\left(\x A\x^{T}=2\ell\right)$.

This point of view suggests the application of \emph{quadratic anticoncentration
inequalities}. Indeed, initially motivated by applications in random
matrix theory \cite{CTV06}, several authors \cite{Ngu12,RV13,Cos13,MNV16}
have studied probabilities of the form $\Pr\left(\g A\g^{T}=x\right)$,
for $\g$ a sequence of \emph{independent} random variables. The general
theme is that if there are many nonzero entries arranged appropriately
in $A$, then this probability is small.

Of course, due to the condition that $\x$ has exactly $k$ ones,
it is not a sequence of independent random variables, but one might
hope that the dependencies are not too severe. For example, $\x$
is in some sense quite similar to the random vector $\x_{\Ber}=\left(\gamma_{1},\dots,\gamma_{n}\right)$
where each $\gamma_{i}$ is independently $\left(k/n\right)$-Bernoulli-distributed\footnote{We say that $\gamma$ has the $p$-Bernoulli distribution if $\Pr\left(\gamma=1\right)=p$
and $\Pr\left(\gamma=0\right)=1-p$ .}. If $A$ has few nonzero entries, then one can prove using a concentration
inequality that $\x_{\Ber}A\x_{\Ber}^{T}$ is likely to be small (and
therefore not equal to $\ell$, unless $\ell$ is itself small). It
is therefore very straightforward to apply a quadratic anticoncentration
inequality to prove a variant of \cref{thm:almost} with $\x_{\Ber}$
in place of $\x$ (meaning that $X_{G,k}$ is the number of edges
in a $\left(k/n\right)$-Bernoulli random set, instead of a uniformly
random set of exactly $k$ vertices). Actually, in general, for any
$r$-uniform hypergraph $G$, the random variable $X_{G,k}$ can be
expressed as a homogeneous degree-$r$ polynomial of $\x$. So, using
a cubic anticoncentration inequality we can similarly give an easy
proof of the ``Bernoulli version'' of \cref{thm:hypergraph-sqrt},
and using the Bonami--Beckner hypercontractive inequality
we can give an easy proof of the ``Bernoulli version'' of \cref{thm:1-eps}.

However, in the setting of this paper, approximating $\x$ with $\x_{\Ber}$
is quite unsatisfactory, because in addition to the ``genuine''
anticoncentration coming from the combinatorial structure of $G$,
there is also spurious anticoncentration arising from fluctuation
in the number of ones in $\x_{\Ber}$. For example, if $G$ is a graph clique
then $\x A\x^{T}$ is constant, while $\x_{\Ber}A\x_{\Ber}^{T}$ is
anticoncentrated purely because the number of vertices in a $\left(k/n\right)$-Bernoulli
random set is itself anticoncentrated.

In the setting of \cref{thm:1-eps} it is straightforward to overcome
this issue: we merely apply a different hypercontractive inequality
in place of the Bonami--Beckner inequality. Despite the
widespread utility of the Bonami--Beckner inequality, the
wider theory of hypercontractive inequalities does not seem to be
well-known in the combinatorics community. In our case the necessary
inequality is essentially due to Lee and Yau \cite{LY98}.

For \cref{thm:almost,thm:hypergraph-sqrt}, we use a coupling argument:
it turns out that there is a natural way to realise the distribution
of $\x$ as a function of a random permutation $\sigma$ and a certain
sequence $\g$ of i.i.d.\ random variables. If we condition on any
outcome of $\sigma$, then $X_{G,k}$ can be viewed as a (non-homogeneous)
degree-$r$ polynomial $f_{\sigma}\left(\g\right)$ of $\g$, to which
we can apply standard anticoncentration inequalities. For a non-homogeneous
polynomial, anticoncentration inequalities tend to give bounds depending
on the nonzero coefficients of maximum degree, so the remaining difficulty
lies in studying the nonzero maximium-degree coefficients in $f_{\sigma}$
(which depend on $\sigma$).

It turns out that these coefficients have a combinatorial interpretation:
for example, if $G$ is a graph (as in \cref{thm:almost}), then the
nonzero degree-2 coefficients in $f_{\sigma}$ in some sense arise
from 4-tuples of vertices $\left(x,x',y,y'\right)$ such that
\[
a_{xy}-a_{xy'}-a_{x'y}+a_{x'y'}\ne0.
\]
In the special case where $\minell=\Omega\left(k^{2}\right)$, we
can use a simple Ramsey-type argument to show that $G$ has $\Omega\left(k^{4}\right)$
such tuples (this turns out to follow from the fact that 2-edge-coloured complete graphs with many blue and red edges have many alternating paths of length 3). This allows us to show that $f_{\sigma}$ is likely
to have many nonzero coefficients, allowing us to deduce \cref{thm:almost}
via a quadratic anticoncentration inequality. For the general case
of \cref{thm:almost} we need to use a more refined anticoncentration
inequality due to Meka, Nguyen and Vu \cite{MNV16} for which it suffices
to find a large matching in an auxiliary graph defined in terms of
the nonzero degree-2 coefficients. In the proof of \cref{thm:almost}
this auxiliary graph will be a random graph depending on $\sigma$.
We will carefully define a greedy procedure that finds the required
matching with high probability.

The situation for hypergraphs is much less straightforward than for
graphs, which is why \cref{thm:hypergraph-sqrt} is so much weaker
than \cref{thm:almost}. In contrast to the graph case, even in the
setting of \cref{thm:hypergraph-sqrt} where $G$ is a 3-uniform hypergraph
with $\minell=\Omega\left(k^{3}\right)$, it may happen that $f_{\sigma}\left(\g\right)$
has very few degree-3 coefficients, which prevents us from directly
applying an anticoncentration inequality. To overcome this, we prove
a variant of the Meka--Nguyen--Vu anticoncentration
inequality which (under certain specific circumstances) allows us
to take coefficients of non-maximum degree into account. We then prove
an approximate classification of 3-uniform hypergraphs $G$ such that
$f\left(\g\right)$ has few nonzero degree-3 coefficients (using a theorem of Fox and Sudakov on ``unavoidable patterns'' and the induced
hypergraph removal lemma), and we study the lower-degree coefficients
of $f$ for all such $G$. This unfortunately involves some slightly
complicated case analysis.

\subsection{Notation}

We use standard asymptotic notation throughout. For functions $f=f\left(n\right)$
and $g=g\left(n\right)$ we write $f=O\left(g\right)$ to mean there
is a constant $C$ such that $\left|f\right|\le C\left|g\right|$,
we write $f=\Omega\left(g\right)$ to mean there is a constant $c>0$
such that $f\ge c\left|g\right|$ for sufficiently large $n$, we
write $f=\Theta\left(g\right)$ to mean that $f=O\left(g\right)$
and $f=\Omega\left(g\right)$, and we write $f=o\left(g\right)$ or
$g=\omega\left(f\right)$ to mean that $f/g\to0$ as $n\to\infty$.
All asymptotics are as $n\to\infty$ unless specified otherwise (specifically,
notation of the form $o_{k}\left(1\right)$ indicates that asymptotics
are as $k\to\infty$).

For a positive integer $n$ , we write $\range n$ to mean the set
$\left\{ 1,\dots,n\right\} $. For a set $S$ we write $\binom S k$ for the collection of all subsets of $S$ of size exactly $k$, and we write $\binom S{\le k}$ for the collection of all subsets of size at most $k$. Less standardly, for a zero-one sequence $\boldsymbol{x}=\left(x_{1},\dots,x_{n}\right)$,
we write $\left|\boldsymbol{x}\right|$ for the number of ones in
$\boldsymbol{x}$. For any sequence $\boldsymbol{x}=\left(x_{1},\dots,x_{n}\right)$,
and any $I\subseteq\range n$, we write $\boldsymbol{x}^{I}$ to mean
the monomial $\prod_{i\in I}x_{i}$.

We also use standard (hyper)graph theoretic notation. In particular,
for a hypergraph $G$ on the vertex set $V$ and a set of vertices
$S\subseteq V$, let $\deg\left(S\right)$ be the number of edges
$e\in E(G)$ such that $S\subseteq e$. Also, we write ``$r$-graph'' as shorthand for ``$r$-uniform hypergraph''.

\section{\label{sec:tools}Probabilistic Tools}

For any $0\le k\le n$ let $\hyp\left(n,k\right)$ be the uniform
distribution on sequences $\boldsymbol{x}\in\left\{ 0,1\right\} ^{n}$
with $\left|\boldsymbol{x}\right|=k$. This is precisely the distribution
of $\x$ as outlined in \cref{sec:outline}. It is sometimes informally
known as the uniform distribution ``on the slice'', and is also
known as the limiting distribution of the Bernoulli--Laplace
model of diffusion. For an $r$-graph $G$ on the vertex
set $\range n$, note that the random variable $X_{G,k}$ can be interpreted
as a homogeneous degree-$r$ polynomial of $\x\in\hyp\left(n,k\right)$:
if $a_{S}=\one_{S\in E\left(G\right)}$ indicates the presence of
an edge $S\subseteq\range n$ in $G$, then we can write
\[
X_{G,k}=\sum_{S}a_{S}\x^{S}.
\]
In this section we collect a number of general results about $\hyp\left(n,k\right)$
which will be useful for our proofs. Some of these are known, and
some are new.

\subsection{Concentration}

There are a number of well-known concentration results which can be
applied to functions of $\hyp\left(n,k\right)$-distributed random
zero-one sequences (see for example \cite[Corollary~2.2]{GIKM17}).
In the proof of \cref{thm:almost} we will use the following ``non-uniform''
concentration inequality, which we were not able to find elsewhere
in the literature.
\begin{lem}
\label{lem:concentration}Consider $f:\left\{ 0,1\right\} ^{n}\to\RR$
such that
\[
\left|f\left(x_{1},\dots,x_{i-1},0,x_{i+1},\dots,x_{n}\right)-f\left(x_{1},\dots,x_{i-1},1,x_{i+1},\dots,x_{n}\right)\right|\le c_{i}
\]
for all $\boldsymbol{x}\in\left\{ 0,1\right\} ^{n}$ and all $i\in\range n$.
Let 
$\x\in\hyp\left(n,k\right)$. Then 
\[
\Pr\left(f\left(\x\right)-\E f\left(\x\right)\ge t\right)\le\exp\left(-\frac{t^{2}}{8\sum_{i=1}^{n}c_{i}^{2}}\right).
\]
\end{lem}

\begin{proof}
We may assume without loss of generality that $c_{1}\ge\dots\ge c_{n}$.
Consider the Doob martingale $Z_{i}=\E\left[f\left(\x\right)\middle|\xi_{1},\dots\xi_{i}\right]$,
so $Z_{0}=\E f\left(\x\right)$ and $Z_{n}=Z_{n-1}=f\left(\x\right)$.
Let $\L\left(x_{1},\dots,x_{i}\right)$ be the conditional distribution
of $\x$ given $\xi_{1}=x_{1},\dots,\xi_{i}=x_{i}$.

We want to show that
\[
\left|\E\left[f\left(\L\left(x_{1},\dots,x_{i-1},0\right)\right)\right]-\E\left[f\left(\L\left(x_{1},\dots,x_{i-1},1\right)\right)\right]\right|\le2c_{i}
\]
for all feasible $x_{1},\dots,x_{i-1}\in\left\{ 0,1\right\} $; this
will imply that $\left|Z_{i}-Z_{i-1}\right|$ is uniformly bounded
by $2c_{i}$, so the desired result will follow from the Azuma--Hoeffding
inequality (see for example \cite[Theorem~2.25]{JLR00}).

If $\x$ is distributed as $\L\left(x_{1},\dots,x_{i-1},0\right)$,
we can change $\xi_{i}$ to 1 and then randomly choose one of the
ones among $\xi_{i+1},\dots,\xi_{n}$ and change it to 0; we thereby
obtain the distribution $\L\left(x_{1},\dots,x_{i-1},1\right)$. This
provides a coupling between $\L\left(x_{1},\dots,x_{i-1},0\right)$
and $\L\left(x_{1},\dots,x_{i-1},1\right)$ that differs in only two
coordinates $i$ and $j>i$, and since $c_{j}\le c_{i}$ this implies
the required bound.
\end{proof}

\subsection{\label{subsec:hypercontractivity}``Weak'' anticoncentration via
hypercontractivity}

The key ingredient in our proof of \cref{thm:1-eps} is the following
``weak'' anticoncentration inequality.
\begin{lem}
\label{lem:hyp-anticoncentration}Let $f$ be an $n$-variable polynomial
of degree $d$, and let $\x\in\hyp\left(n,n/2\right)$. Suppose that
the random variable $f\left(\x\right)$ is not constant. Then, for
any $\ell\in\RR$, we have
\[
\Pr\left(f\left(\x\right)=\ell\right)\le1-e^{-O\left(d\right)}.
\]
\end{lem}

To prove \cref{lem:hyp-anticoncentration} it will suffice to control
the fourth moment of low-degree polynomials of $\hyp\left(n,k\right)$-distributed
random vectors. This is due to the fact (also observed in \cite{AHKT})
that if a random variable has fourth moment comparable to its variance
squared then it is reasonably likely to have fluctuations comparable
to its standard deviation. The following lemma is a corollary of \cite[Lemma~3.2~(i)]{AGK}.
\begin{lem}
\label{lem:fourth-moment-anticoncentration}Let $Z$ be a non-constant real random
variable satisfying $\E Z=0$ and $\E Z^{4}\le b\left(\E Z^{2}\right)^{2}$. Then
for any $\ell\in\RR$, $\Pr\left(Z\ne\ell\right)\ge1/(2^{4/3}b)$.
\end{lem}

Now, to bound the fourth moment of a low-degree polynomial of a $\hyp\left(n,k\right)$-distributed
random vector, we will want a hypercontractive inequality that can
be applied to $\hyp\left(n,k\right)$ analogously to standard applications
of the Bonami--Beckner hypercontractive inequality (see
for example \cite{Odo14}) in discrete Fourier analysis. We will use
a hypercontractive inequality for (the Markov semigroup of) Bernoulli-Laplace
diffusion, which can be deduced (see for example \cite{DS96}) from
a log-Sobolev inequality proved by Lee and Yau \cite{LY98}. We will
present this hypercontractive inequality in a convenient self-contained
form due to Filmus \cite{Fil16}.

First, we need the notion of a \emph{harmonic} polynomial, originally
introduced by Dunkl \cite{Dun76,Dun78}: a polynomial 
$g$ in the
variables $x_{1},\dots,x_{n}$ is said to be harmonic if 
\[
\sum_{i=1}^{n}\frac{\partial g}{\partial x_{i}}=0.
\]

It turns out that every random variable of the form $f\left(\x\right)$,
for $\x\in\hyp\left(n,k\right)$, can be represented in the form $g\left(\x\right)$,
for $g$ a harmonic multilinear polynomial. We will moreover need
the fact that if $f$ is a polynomial of degree $d$, then $g$ also
has degree at most $d$. The following lemma effectively appears as
\cite[Lemma~3.17]{FM16}.
\begin{lem}
\label{lem:harmonic}Let $f$ be an $n$-variable polynomial of degree
$d$, and let $\x\in\hyp\left(n,k\right)$. Then there is a harmonic
multilinear polynomial $g$ such that $f\left(\x\right)=g\left(\x\right)$;
this polynomial has degree at most $\min\left\{ d,k,n-k\right\} $.
\end{lem}

Now, for a harmonic multilinear polynomial $f$, let $f^{=d}$ be
the $d$-th homogeneous part of $f$ consisting of terms with degree
exactly $d$. Each of these parts is ``orthogonal'' in the sense
that for $\x\in\hyp\left(n,k\right)$ and $d\ne d'$ we have $\E\left[f^{=d}\left(\x\right)f^{=d'}\left(\x\right)\right]=0$
(see \cite[Theorem~3.1]{Fil16}). For each $t$, let $H_{t}$ be the\emph{
}operator on $n$-variable harmonic multilinear polynomials defined
as follows. For a harmonic multilinear polynomial $f$, let
\[
H_{t}f=\sum_{i=0}^{n}\exp\left(-t\frac{2i\left(n+1-i\right)}{n\left(n-1\right)}\right)f^{=i}.
\]
We are now in a position to state the promised hypercontractive inequality,
which essentially appears as \cite[Proposition~6.2]{Fil16}.
\begin{prop}
\label{prop:hypercontractive}Let $\x\in\hyp\left(n,pn\right)$ and
let
\[
\rho=-\frac{2}{n\log2\log\left(p\left(1-p\right)\right)}.
\]
Then for any $t\ge0$ and $q\ge2$ satisfying $q-1\le e^{2\rho t}$,
and any $n$-variable harmonic multilinear polynomial 
$g$, we have
\[
\E\left[\left|H_{t}g\left(\x\right)\right|^{q}\right]^{2/q}\le\E g\left(\x\right)^{2}.
\]
\end{prop}

The only reason we need \cref{prop:hypercontractive} is for the following
corollary.
\begin{cor}
\label{cor:hypercontractive-fourth-moment}Let $\x\in\hyp\left(n,n/2\right)$
and let $f$ be an $n$-variable polynomial of degree $d$. Then 
\[
\E f\left(\x\right)^{4}=e^{O\left(d\right)}\left(\E f\left(\x\right)^{2}\right)^{2}.
\]
\end{cor}

\begin{proof}
By \cref{lem:harmonic}, there is a harmonic multilinear polynomial
$g$ of degree at most $d$ such that $f\left(\x\right)=g\left(\x\right)$.
By \cref{prop:hypercontractive} with $p=1/2$, $q=4$ and $t=\Theta\left(n\right)$,
and orthogonality of the different homogeneous parts, we have
\begin{align*}
\left(\E f\left(\x\right)^{4}\right)^{1/2} & \le\E\left(H_{-t}f\left(\x\right)\right)^{2}\\
 & =\sum_{i=0}^{d}e^{O\left(i\right)}\E\left(f^{=i}\left(\x\right)\right)^{2}\\
 & =e^{O\left(d\right)}\E\left(f\left(\x\right)\right)^{2}.\tag*{\qedhere}
\end{align*}
\end{proof}
\cref{lem:hyp-anticoncentration} is now a direct consequence of \cref{cor:hypercontractive-fourth-moment}
and \cref{lem:fourth-moment-anticoncentration}, applied to $f\left(\x\right)-\E f\left(\x\right)$.

\subsection{Coupling}
\label{subsec:coupling}
Many standard probabilistic tools are designed to work for product
distributions, where independence can be exploited. Although $\hyp\left(n,k\right)$
is not a product distribution, there is a well-known way to approximate
$\hyp\left(n,k\right)$ with a product of Bernoulli-distributed random
variables, and for many purposes these distributions can be considered
essentially equivalent (see for example \cite[Corollary~1.16]{JLR00}
and the invariance principles in \cite{FKMW16,FM16}). However, for
the purposes of proving \cref{thm:almost,thm:hypergraph-sqrt} this
kind of approximation is too coarse. Instead we will use a non-standard
coupling between $\hyp\left(n,n/2\right)$ and $\Rad^{n/2}$, where
$\Rad$ is the Rademacher distribution (that is, the uniform distribution
on $\left\{ -1,1\right\} $). The following observation essentially
appears in the proof of \cite[Proposition~4.10]{LLTTY17} (a similar
coupling also appears in 
\cite[p.~15]{AHKT}).
\begin{fact}
\label{fact:coupling}The distribution $\x\in\hyp\left(n,n/2\right)$
can be obtained as follows. Let $\sigma$ be a uniformly random permutation
of $\range n$ and let $\boldsymbol{\gamma}\in\Rad^{n/2}$ be a sequence
of $n/2$ i.i.d.\ Rademacher random variables. Then set $\xi_{\sigma\left(i\right)}=1$
for all $i$ such that $\gamma_{i}=1$ and set $\xi_{\sigma\left(i+n/2\right)}=1$
for all $i$ such that $\gamma_{i}=-1$. For all other indices $j$
set $\xi_{j}=0$.
\end{fact}

In order to use \cref{fact:coupling}, we need to translate polynomials
of $\x\in\hyp\left(n,n/2\right)$ into polynomials of $\boldsymbol{\gamma}\in\Rad^{n/2}$.
\begin{lem}
\label{lem:coefficients}Consider a random variable $X$ of the form
\[
X=\sum_{S\in \binom{\range n}d}a_{S}\x^{S},
\]
where $\x\in\hyp\left(n,n/2\right)$. Under the coupling in \cref{fact:coupling},
$X$ is a function of $\g,\sigma$. If we condition on any outcome
of $\sigma$, then $X$ is a multilinear polynomial in the $\gamma_{i}$
with degree at most $d$. Consider a subset $I\subseteq\range{n/2}$
of size at least $d-1$, and write $I=\left\{ i_{1},\dots,i_{q}\right\} $;
then the coefficient $g_{I}$ of $\g^{I}$ is
\[
\frac{1}{2^{d}}\sum_{\bb\in\left\{ 0,1\right\} ^{q}}\left(-1\right)^{\left|\bb\right|}a\left(\left\{ \sigma\left(i_{j}+\b_{j}\frac{n}{2}\right):1\le j\le q\right\} \right),
\]
where for $R\subseteq\range n$, $a\left(R\right)$ is the sum of
all $a_{S}$ with $S\supseteq R$.
\end{lem}

\begin{proof}
Given a permutation $\sigma$ of $\range n$, define the functions
$\alpha:\range n\to\left\{ 0,1\right\} $ and $\tilde{\sigma}:\range n\to\range{n/2}$
by
\[
\left(\alpha\left(x\right),\tilde{\sigma}\left(x\right)\right)=\begin{cases}
\left(0,\sigma^{-1}\left(x\right)\right) & \sigma^{-1}\left(x\right)\le n/2\\
\left(1,\sigma^{-1}\left(x\right)-n/2\right) & \text{otherwise}.
\end{cases}
\]
Also, for $S\subseteq\range n$, let $\left|\alpha\left(S\right)\right|$
be the number of $x\in S$ for which $\alpha\left(x\right)=1$.

Now, observe that $\xi_{x}=\left(1+\left(-1\right)^{\alpha\left(x\right)}\gamma_{\tilde{\sigma}\left(x\right)}\right)/2$.
We may write
\begin{align*}
X & =\frac{1}{2^{d}}\sum_{S\in \binom{\range n}d}a_{S}\prod_{x\in S}\left(1+\left(-1\right)^{\alpha\left(x\right)}\gamma_{\tilde{\sigma}\left(x\right)}\right)\\
 & =\frac{1}{2^{d}}\sum_{S\in \binom{\range n}d}a_{S}\sum_{R\subseteq S}\left(-1\right)^{\left|\alpha\left(R\right)\right|}\prod_{x\in R}\gamma_{\tilde{\sigma}\left(x\right)}.
\end{align*}
Consider any $I=\left\{ i_{1},\dots,i_{q}\right\} \subseteq\range{n/2}$
with $q\ge d-1$. We have $\g^{I}=\prod_{x\in R}\gamma_{\tilde{\sigma}\left(x\right)}$
if and only if $R$ is of the form $\left\{ \sigma\left(i_{j}+\b_{j}n/2\right):1\le j\le q\right\} $
for some $\bb\in\left\{ 0,1\right\} ^{q}$, in which case
$\left|\alpha\left(R\right)\right|=\left|\bb\right|$.
(If $q<d-1$ then there are other possibilities for $R$ due to the fact that $\gamma_{i}^{2}=1$). The desired result follows.
\end{proof}

As an illustration of \cref{lem:coefficients}, we consider the special
case where $X$ is of the form $\sum_{i<j}a_{ij}\xi_{i}\xi_{j}$,
for $\x\in\hyp\left(n,n/2\right)$. If $G$ is an $n$-vertex graph
with adjacency matrix $\left(a_{ij}\right)$, then this random variable
has precisely the distribution of $X_{G,n/2}$. Under the coupling
in \cref{fact:coupling}, if we condition on any outcome of $\sigma$,
then $X$ is a quadratic polynomial in the $\gamma_{i}$, and the
coefficient of $\gamma_{i}\gamma_{j}$ is
\[
\frac{1}{4}\left(a_{\sigma\left(i\right)\sigma\left(j\right)}-a_{\sigma\left(i\right)\sigma\left(j+n/2\right)}-a_{\sigma\left(i+n/2\right)\sigma\left(j\right)}+a_{\sigma\left(i+n/2\right)\sigma\left(j+n/2\right)}\right).
\]

\subsection{Polynomial Anticoncentration}

In the proof of \cref{thm:almost,thm:hypergraph-sqrt} we will use
an anticoncentration inequality for polynomials of Rademacher random
variables proved by Meka, Nguyen and Vu \cite{MNV16}. For $\boldsymbol{x}=\left(x_{1},\dots,x_{n}\right)$,
consider a degree-$d$ polynomial 
\[
f\left(\boldsymbol{x}\right)=\sum_{S\in \binom{\range n}{\le d}}f_{S}\boldsymbol{x}^{S}
\]
in $\boldsymbol{x}$. The \emph{rank} of $f$ is the size of the largest
matching in the $d$-uniform hypergraph on the vertex set $\range n$
obtained by putting an edge $S\subseteq\range n$ whenever $\left|S\right|=d$
and $f_{S}\ne0$. The following theorem is a direct corollary of \cite[Theorem~1.6]{MNV16}. 
\begin{thm}
\label{thm:MNV}Fix $d\in\NN$ and let $\g\in\Rad^{n}$. Let $f$
be a degree-$d$ polynomial with rank $r$. Then for any $\ell\in\RR$,
\[
\Pr\left(f\left(\g\right)=\ell\right)\le\frac{\log^{O\left(1\right)}\left(r\right)}{\sqrt{r}}.
\]
\end{thm}

A drawback of \cref{thm:MNV} is that it ignores coefficients of $f$
that are not of maximum degree. We deduce the following result, which
allows us to take these coefficients into account under certain circumstances.
We will need this for the proof of \cref{thm:hypergraph-sqrt}.
\begin{cor}
\label{cor:second-degree-MNV}Fix $d\in\NN$, let $\g\in\Rad^{n}$,
and consider a degree-$d$ polynomial
\[
f\left(\boldsymbol{x}\right)=\sum_{S\in \binom{\range n}{\le d}}f_{S}\boldsymbol{x}^{S}.
\]
Let $m_{d}=\max\left\{ \left|f_{S}\right|:\left|S\right|=d\right\} $
be the maximum coefficient of degree $d$, and let $H'$ be the $\left(d-1\right)$-uniform
hypergraph with edge set $\left\{ S:\left|S\right|=d-1,\left|f_{S}\right|\ge rm_{d}\right\} $.
Suppose that $H'$ has a matching of size $r$. Then for any $\ell\in\RR$,
\[
\Pr\left(f\left(\g\right)=\ell\right)\le\frac{\log^{O\left(1\right)}\left(r\right)}{\sqrt{r}}.
\]
\end{cor}

\begin{proof}
Let $H$ be the $d$-uniform hypergraph used to define the rank of
$f$, with an edge for every nonzero degree-$d$ coefficient. If $H$
has a matching of size $r/(2d)$ then we are done by \cref{thm:MNV}. Otherwise,
$H$ has an independent set $I$ of size larger than $n-r/2$. Condition
on any outcome of the values $\gamma_{i}$ for $i\notin I$. Now,
$f\left(\g\right)$ can be expressed as a polynomial $g\left(\left(\gamma_{i}\right)_{i\in I}\right)$
of the remaining $\gamma_{i}$, depending on the values of $\gamma_{i}$
for $i\notin I$. This polynomial has degree at most $d-1$. Specifically,
for $S\subseteq I$ with size $d-1$, the coefficient of $S$ in $g$
is 
\[
f_{S}+\sum_{i\notin I}f_{S\cup\left\{ i\right\} }\gamma_{i}.
\]
If $f_{S}\ge rm_{d}$, then this coefficient is nonzero (in fact, $f_{S}\ge (r/2)m_{d}$ suffices), so each
edge of $H'[I]$ corresponds to a nonzero degree-$\left(d-1\right)$
coefficient in $g$. Moreover, by assumption $H'[I]$ has a matching of size at least $r-(n-|I|)\ge r/2$. Therefore, $g$ has rank at least
$r/2$, so the desired result again follows from \cref{thm:MNV}.
\end{proof}

\section{\label{sec:almost}Anticoncentration in graphs}

In this section we present the proof of \cref{thm:almost}. It suffices to prove that $\Pr\left(X_{G,k}=\ell\right)=O\left(\log^{O\left(1\right)}\left(\minell/k\right)/\sqrt{\minell/k}\right)$
for every graph $G$ with $2k$ vertices, because $I\left(n,k,\ell\right)$
is a monotone nonincreasing function of $n$. So, let $G$ be a graph
on the vertex set $\range n$, for $n=2k$. Let $X=X_{G,k}$. We may
assume that $e\left(G\right)\le\binom{n}{2}/2$,
by taking the complement of $G$ if necessary. We express $X$ in
the form $X\left(\x\right)=\sum_{1\le x<y\le n}a_{xy}\xi_{x}\xi_{y}$,
where the $a_{xy}$ are the entries of the adjacency matrix of $G$
and $\x\in\hyp\left(n,n/2\right)$.

Note that $\E X\approx e\left(G\right)/4$. We first want to use \cref{lem:concentration}
to show that if $e\left(G\right)$ is not at least of the same order as $\ell$
then $\Pr\left(X_{G,k}=\ell\right)$ is very small.
\begin{claim}
\label{claim:many-edges}For any constant $\varepsilon>0$, if $\ell\ge\left(1+\varepsilon\right)\E X$ or $\ell\le\left(1-\varepsilon\right)\E X$
then 
\[
\Pr\left(X=\ell\right)\le\exp\left(-\Omega\left(\frac{\varepsilon^{2}\ell}{k}\right)\right).
\]
\end{claim}

\begin{proof}
Note that $X$ is of the form required to apply \cref{lem:concentration},
with $c_{x}=\deg\left(x\right)$. Then
\begin{align*}
\Pr\left(X=\ell\right) & \le\exp\left(-\Omega\left(\frac{\varepsilon^{2}e\left(G\right)^{2}}{\sum_{x=1}^{n}\deg\left(x\right)^{2}}\right)\right)\\
 & \le\exp\left(-\Omega\left(\frac{\varepsilon^{2}e\left(G\right)^{2}}{n^{2}\left(e\left(G\right)/n\right)}\right)\right)\\
 & \le\exp\left(-\Omega\left(\frac{\varepsilon^{2}\ell}{k}\right)\right),
\end{align*}
where in the second inequality we have used the upper bound on $\sum_{x=1}^{n}\deg\left(x\right)^{2}$ obtained by taking $\deg(x)=n$ for as many $x$ as possible, and in the third inequality we have used the fact that $\ell\le\left(1-\varepsilon\right)\E X=O\left(e\left(G\right)\right)$.
\end{proof}
From now on we will assume that $e\left(G\right)=\Omega\left(\ell\right)$, which also implies that $\ell=\Theta(\ell^*)$.
Now, let $\sigma$ be a uniformly random permutation of $\range n$,
and let $H$ be the (random) graph on the vertex set $\range k$ with
an edge between $i$ and $j$ if
\[
a_{\sigma\left(i\right)\sigma\left(j\right)}-a_{\sigma\left(i\right)\sigma\left(j+k\right)}-a_{\sigma\left(i+k\right)\sigma\left(j\right)}+a_{\sigma\left(i+k\right)\sigma\left(j+k\right)}\ne0.
\]
The heart of the proof of \cref{thm:almost} is the following claim.
\begin{claim}
\label{claim:matching}
The graph $H$ has a matching of size $\Omega\left(\ell/k\right)$, with probability $1-O\left(k/\ell\right)$.
\end{claim}

Before proving \cref{claim:matching}, we will show how it implies \cref{thm:almost}.

\begin{proof}[Proof of \cref{thm:almost}]
Let $\mathcal{E}$ be the event that $H$ has a matching of size $\Omega(\ell/k)$. We learn from \cref{lem:coefficients} (see the discussion at the end of \cref{subsec:coupling}) that $X$ is a quadratic polynomial in $\gamma \in \Rad^{n/2}$, and the coefficient of $\gamma_i \gamma_j$ is $\frac14 \left( a_{\sigma\left(i\right)\sigma\left(j\right)}-a_{\sigma\left(i\right)\sigma\left(j+k\right)}-a_{\sigma\left(i+k\right)\sigma\left(j\right)}+a_{\sigma\left(i+k\right)\sigma\left(j+k\right)}\right).$ Hence the rank of $X$ (as a polynomial in the $\gamma_i$) is equal to the size of a maximum matching in $H$. Thus $\Pr(X=\ell\big | \mathcal{E}) \le \log^{O(1)}(\ell/k)\sqrt{\frac{k}{\ell}}$, by \cref{thm:MNV}. Combined with \cref{claim:matching}, we obtain 
\[
\Pr(X=\ell) \le \Pr(\overline{\mathcal{E}})+\Pr(X=\ell\big | \mathcal{E}) \le \log^{O(1)}(\ell/k)\sqrt{\frac{k}{\ell}}. \qedhere
\]
\end{proof}

In order to prove \cref{thm:almost}, it remains to show \cref{claim:matching}. As a warm-up, we first sketch the proof of \cref{claim:matching} in the regime where $\ell=\Omega\left(k^{2}\right)$. 
A key observation is that if $a_{vw}=a_{v'w'}\ne a_{v'w}$
(that is, the path $vwv'w'$ alternates between edges and non-edges)
then $a_{vw}-a_{vw'}-a_{v'w}+a_{v'w'}\ne0$. That is to say, edges
in $H$ arise from alternating paths of length 3 in $G$.  
When $\ell=\Omega\left(k^{2}\right)$, we can show that $G$ has $\Omega\left(k^{4}\right)$ alternating 3-paths. Roughly speaking, the reason is that $G$ can be divided
into two parts $V_{1}$ and $V_{2}$ such that all vertices in $V_{1}$
have reasonably high degree and all vertices in $V_{2}$ have reasonably
high non-degree. If there were many non-edges in $V_{1}$ or many
edges in $V_{2}$ this would immediately give us many alternating
3-paths. Otherwise $V_{1}$ is almost a clique and $V_{2}$ is almost
an independent set, in which case almost every pair of edges between
$V_{1}$ and $V_{2}$ gives rise to an alternating 3-path through
$V_{2}$, and almost every pair of non-edges between $V_{1}$ and
$V_{2}$ gives rise to an alternating 3-path through $V_{1}$, and
there must be many alternating 3-paths of at least one of these two
types. Now, if $G$ has $\Omega\left(k^{4}\right)$ alternating 3-paths,
it follows from a concentration inequality that $H$ is very likely
to have $\Omega\left(k^{2}\right)$ edges, and hence a matching
of size $k$.

In the general case this simplistic approach does not suffice, and
the way we find our matching in $H$ will differ slightly depending
on the structure of $G$. Let $U\subseteq\range n$ be the set of
``high-degree'' vertices with degree at least $0.9 n$. We divide the proof of \cref{claim:matching} into two cases: the case when many edges are incident to $U$ will be handled in \cref{subsec:case-1}, and the case where many edges avoid $U$ will be treated in \cref{subsec:case-2}.

\subsection{\label{subsec:case-1}Case 1: many edges are incident to the high-degree
vertices}

First, consider the case where $e\left(G\right)/2$ edges are incident
to $U$. In this case, $2k\left|U\right|\ge e\left(G\right)/2$, so
$\left|U\right|=\Omega\left(\ell/k\right)$. We can in fact assume that $|U|\ge 3$, because if $\ell=O(k)$ the statement of \cref{thm:almost} is trivial. Now consider the following
procedure to iteratively build a matching $M$ in $H$. At step $t$,
let $V_{t}$ be the set of vertices $v\in\range n$ such that the
value of $\sigma^{-1}\left(v\right)$ has not yet been revealed, choose
any distinct $u,w\in U\cap V_{t}$, and reveal the values of $i=\sigma^{-1}\left(u\right)$
and $j=\sigma^{-1}\left(w\right)$. If $i,j\le k$ and the values
of $\sigma\left(i+k\right)$ and $\sigma\left(j+k\right)$, have not
already been revealed, reveal them, and if we find that $\left\{ i,j\right\} $
is an edge in $H$ then we add $\left\{ i,j\right\} $ to $M$.

The above procedure can continue while $4t+2\le\left|U\right|$ (that
is, $t\le\left(\left|U\right|-2\right)/4$). Let $T=\min\left\{ \left(\left|U\right|-2\right)/4,\,0.01n\right\} =\Omega\left(\ell/k\right)$.
We claim that every step $t\le T$ has probability $\Omega\left(1\right)$
of successfully adding an edge to $M$.
\begin{claim}
\label{claim:many-high-prob}For any $t\le T=\Omega\left(\ell/k\right)$,
condition on any outcome of the information revealed before step
$t$. Then the probability that an edge is added to $M$ in step $t$
is at least $0.02$.
\end{claim}

\begin{proof}
Let $Q_{t}$ be the set of indices $q\in\range k$ such that $\sigma\left(q\right)$ or $\sigma\left(q+k\right)$ have already been revealed in previous steps (that is, $\sigma\left(q\right)\notin V_{t}$
or $\sigma\left(q+k\right)\notin V_{t}$). Observe that $\left|Q_{t}\right|\le8T$.
Now, reveal the values of $i=\sigma^{-1}\left(u\right)$ and $j=\sigma^{-1}\left(w\right)$.
The probability that $i,j\le k$ and $\sigma\left(i+k\right),\sigma\left(j+k\right)\in V_{t}$
is at least 
\begin{equation}
\frac{k-\left|Q_{t}\right|}{n}\cdot\frac{k-\left|Q_{t}\right|-1}{n}\ge\left(\frac{k-8T-1}{n}\right)^{2}\ge0.2.\label{eq:many-high-condition}
\end{equation}
Condition on such an outcome of $i,j$. Note that $\left|V_{t}\right|\ge n-4t\ge0.9n$,
and recall that $u$ and $w$ have degree at least $0.9n$
(as vertices in $U$). It follows that $\left|N_{V_{t}}\left(u\right)\right|,\left|N_{V_{t}}\left(w\right)\right|\ge0.8n$.
Let $P\subseteq N_{V_{t}}\left(w\right)\times N_{V_{t}}\left(u\right)$
be the set of distinct ordered pairs of vertices $\left(u',w'\right)$
with $u'\in N_{V_{t}}\left(w\right)\setminus\left\{ u\right\} $
and $w'\in N_{V_{t}}\left(u\right)\setminus\left\{ w\right\} $,
so that $\left|P\right|\ge(0.8n-1)\left(0.8n-2\right)\ge0.6n^{2}$.
Recall that we are assuming $e\left(G\right)\le\binom{n}{2}/2$, so
at most $n^{2}/2$ of the pairs in $P$ are edges of $G$. Let $P'$
be the set of pairs of $P$ which are not edges, so that $\left|P'\right|\ge0.1n^{2}$.
Observe that if $\left(\sigma\left(i+k\right),\sigma\left(j+k\right)\right)\in P'$
then the vertices $\sigma\left(i\right),\sigma\left(j+k\right),\sigma\left(i+k\right),\sigma\left(j\right)$
form an alternating path, so
\[
a_{\sigma\left(i\right)\sigma\left(j\right)}-a_{\sigma\left(i\right)\sigma\left(j+k\right)}-a_{\sigma\left(i+k\right)\sigma\left(j\right)}+a_{\sigma\left(i+k\right)\sigma\left(j+k\right)}=a_{\sigma\left(i\right)\sigma\left(j\right)}-2\ne0,
\]
meaning that $\left\{ i,j\right\} $ is an edge of $H$ and can be
added to $M$. The probability of this is at least $\left|P'\right|/n^{2}\ge0.1$.
Recall that this is a conditional probability, so we multiply by \cref{eq:many-high-condition}
for the desired result.
\end{proof}
\cref{claim:many-high-prob} implies that the eventual size of $M$ stochastically
dominates the binomial distribution $\Bin\left(T,0.02\right)$, so
by the Chernoff bound we have $\left|M\right|\ge0.01T$
with probability $1-e^{-\Omega\left(\ell/k\right)}$.

\subsection{\label{subsec:case-2}Case 2: many edges avoid the high-degree vertices}

Let $\overline{U}=\range n\setminus U$; it remains to consider the
case where $\overline{U}$ induces at least $e\left(G\right)/2=\Omega\left(\ell\right)$
edges. Observe that in $G\left[\overline{U}\right]$ we can greedily
find a matching of size at least $s:=e\left(G\left[\overline{U}\right]\right)/k=\Omega\left(\ell/k\right)$;
let $S\subseteq\range n^{2}$ be such a matching. Now consider the
following procedure to iteratively build a matching $M$ in $H$.
At step $t$, let $V_{t}$ be the set of vertices $v\in\range n$
such that the value of $\sigma^{-1}\left(v\right)$ has not yet been
revealed, choose any $\left(u,v\right)\in S\cap V_{t}^{2}$, and reveal
the values of $i=\sigma^{-1}\left(u\right)$ and $j=\sigma^{-1}\left(w\right)$.
If $i,j\le k$ and the values of $\sigma\left(i+k\right)$ and $\sigma\left(j+k\right)$,
have not already been revealed, reveal them, and if we find that $\left\{ i,j\right\} $
is an edge in $H$ then we add $\left\{ i,j\right\} $ to $M$.

The above procedure can continue while $4t+1\le\left|S\right|$ (that
is, $t\le\left(\left|S\right|-1\right)/4$). Let $T=\min\left\{ \left(\left|S\right|-1\right)/4,\,0.01n\right\} =\Omega\left(\ell/k\right)$.
As in \cref{subsec:case-1}, we claim that every step $t\le T$ has
probability $\Omega\left(1\right)$ of successfully adding an edge
to $M$.
\begin{claim}
\label{claim:many-avoid-prob}For any $t\le T=\Omega\left(\ell/k\right)$,
condition on any outcome of the information revealed before step
$t$. Then the probability that an edge is added to $M$ in step $t$
is at least $0.0004$.
\end{claim}

\begin{proof}
Reveal the values of $i=\sigma^{-1}\left(u\right)$ and $j=\sigma^{-1}\left(w\right)$. As in Case 1, the probability that $i,j\le k$ and $\sigma\left(i+k\right),\sigma\left(j+k\right)\in V_{t}$
is at least $\left(\left(k-8T-1\right)/n\right)^{2}\ge0.2$. Condition
on such an outcome of $i,j$. Note that $\left|V_{t}\right|\ge n-4t\ge0.95n$,
and recall that $u$ and $w$ have degree at most $0.9n$,
so $\left|V_{t}\setminus N\left(u\right)\right|,\left|V_{t}\setminus N\left(w\right)\right|\ge0.05n$.
Let $P\subseteq V_{t}^{2}$ be the set of distinct ordered pairs of
vertices $\left(u',w'\right)$ with $u'\in\left(V_{t}\setminus N\left(u\right)\right)\setminus\left\{ u,w\right\} $
and $w'\in\left(V_{t}\setminus N\left(w\right)\right)\setminus\left\{ u,w\right\} $,
so that $\left|P\right|\ge(0.05n-2)\left(0.05n-3\right)\ge0.002n^{2}$.
Observe that if $\left(\sigma\left(j+k\right),\sigma\left(i+k\right)\right)\in P$
then the vertices $\sigma\left(i+k\right),\sigma\left(j\right),\sigma\left(i\right),\sigma\left(j+k\right)$
form an alternating path and
\[
a_{\sigma\left(i\right)\sigma\left(j\right)}-a_{\sigma\left(i\right)\sigma\left(j+k\right)}-a_{\sigma\left(i+k\right)\sigma\left(j\right)}+a_{\sigma\left(i+k\right)\sigma\left(j+k\right)}=1+a_{\sigma\left(i+k\right)\sigma\left(j+k\right)}\ne0,
\]
meaning that $\left\{ i,j\right\} $ is an edge of $H$ and can be
added to $M$. The probability of this is at least $\left|P\right|/n^{2}\ge0.002$.
\end{proof}
As in \cref{subsec:case-1}, it follows that the eventual size of $M$
stochastically dominates the binomial distribution $\Bin\left(T,0.0004\right)$,
so by the Chernoff bound we have $\left|M\right|\ge0.0002T$
with probability $1-e^{-\Omega\left(\ell/k\right)}$.

\section{\label{sec:1-eps}``Weak'' anticoncentration in hypergraphs}

\cref{thm:1-eps} will be an almost immediate consequence of \cref{lem:hyp-anticoncentration}.
The only combinatorial fact we need is as follows.
\begin{lem}
\label{lem:non-constant}The following holds for any $r$. Let $G$
be an $r$-graph on $2k$ vertices. Then either $G$
induces a clique on $k$ vertices, or an independent set on $k$ vertices,
or else it induces two $k$-vertex subgraphs with different numbers
of edges.
\end{lem}

\begin{proof}
We can assume $k\ge r$, because otherwise $G$ trivially induces
an independent set on $k$ vertices. We will assume that every $k$-vertex
subset of $G$ induces the same number of edges, and prove that $G$
is a clique or independent set.

We claim that for all $s\le k$, $\deg\left(S\right)$ takes a constant
value among vertex subsets $S$ of size $s$. This will
imply the desired result, because if $S$ has $r$ vertices, then
$\deg\left(S\right)$ is either zero or one, depending on whether
$S$ is an edge in $G$. We prove our desired claim by induction on
$s$, so assume it holds for all sizes less than some $s$. For $i<s$,
let $d_{i}$ be the common value of $\deg\left(S\right)$ among $S$
of size $i$.

Now, let $\st(v)$ be the set of edges of $G$ which contain a vertex $v$. For a set $S$ of $s$ vertices, by the inclusion-exclusion principle, the number of edges of $G$
which intersect $S$ is
\begin{align*}
\left|\bigcup_{v\in S}\st\left(v\right)\right| & =\sum_{i=1}^{s}\left(-1\right)^{i+1}\left(\sum_{S'\in \binom{S}{i}}\;\left|\bigcap_{v\in S'}\st\left(v\right)\right|\right)
=\sum_{i=1}^{s-1}\left(-1\right)^{i+1}\binom{s}{i}d_{i}+\left(-1\right)^{s+1}\deg\left(S\right).
\end{align*}
So, if we had $S_{1},S_{2}\subseteq V$ with $\left|S_{1}\right|=\left|S_{2}\right|=s$
and $\left(-1\right)^{s+1}\deg\left(S_{1}\right)<\left(-1\right)^{s+1}\deg\left(S_{2}\right)$,
it would imply that $e\left(V\setminus S_{1}\right)>e\left(V\setminus S_{2}\right)$.
Since we are assuming $k\ge s$ and $\left|V\right|=2k$, we would
have $\left|V\setminus S_{1}\right|=\left|V\setminus S_{2}\right|\ge k$,
so by averaging there would be $k$-vertex subsets $U_{1}\subseteq V\setminus S_{1}$
and $U_{2}\subseteq V\setminus S_{2}$ such that $e\left(U_{1}\right)>e\left(U_{2}\right)$.
This would be a contradiction.
\end{proof}
Now we can prove \cref{thm:1-eps}. Similarly to the proof of \cref{thm:almost},
by monotonicity it suffices to show that $\Pr\left(X_{G,k}=\ell\right)\le1-\varepsilon$,
for some $\varepsilon$ depending only on $r$, whenever $G$ is an
$r$-graph on $n=2k$ vertices. If $X_{G,k}$ is identically
equal to $0$ or $\binom{n}{k}$ then we are done, and otherwise,
by \cref{lem:non-constant}, $X_{G,k}$ must be supported on at least
two values (meaning that it is not a constant).

Then, note that we can express $X_{G,k}$ in the form $f\left(\x\right)$,
for $f$ a polynomial of degree at most $r$ and $\x\in\hyp\left(n,k\right)$.
The desired result therefore follows from \cref{lem:hyp-anticoncentration}.

\section{\label{sec:hypergraph-sqrt}Anticoncentration in dense 3-graphs}

In this section we prove \cref{thm:hypergraph-sqrt}. Let $G$ be a 3-graph on the vertex set $\range n$, for $n=2k$; as
in \cref{sec:almost}, it suffices to prove that $\Pr\left(X_{G,k}=\ell\right)\le\log^{O\left(1\right)}k/\sqrt{k}$.
Also, with the same arguments as in \cref{claim:many-edges}, we can
assume that $\min\left\{ e\left(G\right),\binom{n}{3}-e\left(G\right)\right\} =\Omega\left(n^{3}\right)$.

Now, we express $X=X_{G,k}$ as 
\[
\sum_{1\le x<y<z\le n}a_{xyz}\xi_{x}\xi_{y}\xi_{z},
\]
where $a_{xyz}=\one_{\left\{ x,y,z\right\} \in E\left(G\right)}$
and $\x\in\hyp\left(n,n/2\right)$. One might hope that \cref{conj:hypergraphs}
would follow from a 3-graph generalisation of the arguments in \cref{thm:almost}.
Indeed, it would suffice to show that $G$ has $\Omega\left(n^{6}\right)$
``good'' 6-tuples of vertices $\left(x,x',y,y',z,z'\right)$ such
that
\begin{equation}
a_{xyz}-a_{xyz'}-a_{xy'z}-a_{x'yz}+a_{xy'z'}+a_{x'yz'}+a_{x'y'z}-a_{x'y'z'}\label{eq:3-form}
\end{equation}
is nonzero. Unfortunately, in contrast with the 2-uniform case in
\cref{thm:almost}, there exist 3-graphs $G$ satisfying $\min\left\{ e\left(G\right),\binom{n}{3}-e\left(G\right)\right\} =\Omega\left(n^{3}\right)$
which have no good 6-tuple at all. It will be useful to classify all
such 3-graphs, which we will do after making some definitions. Fix
a set $\left\{ x,x',y,y',z,z'\right\} $ of size 6, and let $\F$
be the set of all 3-graphs on this vertex set such that the expression
in \cref{eq:3-form} is nonzero. We say that a 3-graph is \emph{$\F$-free}
if it induces no 3-graph from $\F$. Also, we define a family of $\F$-free
3-graphs as follows. For two disjoint sets $A$ and $B$ and a set
of disjoint pairs $M\subseteq A\times B$, let $G_{A,B,M}$ be the
3-graph on the vertex set $A\cup B$, whose edges are the triples
which intersect both $A$ and $B$, except those triples which include
a pair from $M$.
\begin{lem}
\label{lem:3-classification}Consider an $\F$-free $n$-vertex 3-graph
$G$ with $\min\left\{ e\left(G\right),\binom{n}{3}-e\left(G\right)\right\} =\Omega\left(n^{3}\right)$.
Then, provided $n$ is sufficiently large, $G$ or its complement
is of the form $G_{A,B,M}$, for some partition $A\cup B$ of the
vertex set of $G$ and some set of disjoint pairs $M\subseteq A\times B$.
\end{lem}

The proof of \cref{lem:3-classification} involves some somewhat complicated
casework, so we defer it to \cref{subsec:3-classification}.

Now, under the coupling in \cref{fact:coupling}, $X$ is a function
of a random permutation $\sigma:\range n\to\range n$ and a random
vector $\g\in\Rad^{n/2}$. By \cref{lem:coefficients}, for any outcome
of $\sigma$, $X$ is a polynomial in the $\gamma_{i}$ of degree
at most 3. The coefficient $g_{ijq}$ of $\gamma_{i}\gamma_{j}\gamma_{q}$
is
\[
\sum_{\bb\in\left\{ 0,1\right\} ^{3}}\left(-1\right)^{\left|\bb\right|}a_{ \sigma\left(i+k\b_{1}\right)\,\sigma\left(j+k\b_{2}\right)\,\sigma\left(q+k\b_{3}\right) }
\]
(note that $|g_{ijq}|\le 4$) and the coefficient $g_{ij}$ of $\gamma_{i}\gamma_{j}$ is
\[
\deg\left(\sigma\left(i\right),\sigma\left(j\right)\right)-\deg\left(\sigma\left(i+k\right),\sigma\left(j\right)\right)-\deg\left(\sigma\left(i\right),\sigma\left(j+k\right)\right)+\deg\left(\sigma\left(i+k\right),\sigma\left(j+k\right)\right).
\]
Let $H$ be the random 3-graph on the vertex set $\range{n/2}$ with
an edge $\left\{ i,j,q\right\} $ whenever $g_{ijq}\ne0$. First suppose that $G$ contains $\Omega\left(n^{6}\right)$ induced
subgraphs from $\F$, and let $N=\Omega\left(n^{6}\right)$ be the
number of ordered 6-tuples $\left(x,x',y,y',z,z'\right)$ such that
the expression in \cref{eq:3-form} is nonzero. Then 
\[
\E e\left(H\right)=\frac n2\left(\frac n2-2\right)\left(\frac n2-4\right)\frac{N}{n\left(n-1\right)\dots\left(n-5\right)}=\Theta\left(n^{3}\right)
\]
Also, note that modifying $\sigma$ by a transposition changes $e\left(H\right)$
by at most $2\binom{n}{2}$. By a McDiarmid-type concentration inequality
for random permutations (see for example \cite[Section~3.2]{McD98}),
we therefore have
\[
\Pr\left(e\left(H\right)\le\E e\left(H\right)/2\right)=\exp\left(-\Omega\left(\frac{\left(n^{3}\right)^{2}}{n\cdot\binom{n}{2}^{2}}\right)\right)=e^{-\Omega\left(n\right)}.
\]
But if $e\left(H\right)\ge\E e\left(H\right)/2=\Omega\left(n^{3}\right)$
then we can greedily find a matching of size $\Omega\left(n\right)$
in $H$, and \cref{thm:MNV} finishes the proof.

It remains to consider the case where $G$ contains $o\left(n^{6}\right)$
induced subgraphs from $\F$. In this case, by the induced hypergraph
removal lemma (see \cite[Theorem~6]{RS09}), we can add and remove
$o\left(n^{3}\right)$ edges from $G$ to obtain a 3-graph with no induced subgraphs from $\F$. By \cref{lem:3-classification},
we can assume this 3-graph is of the form $G_{A,B,M}$. Note that only $O(|M|n)=o(n^3)$ edges of $G_{A,B,M}$ can involve a pair in $M$, so we can actually obtain $G':=G_{A,B,\emptyset}$ by adding and removing
$o\left(n^{3}\right)$ edges from $G$. Recall that $\min\left\{ e\left(G\right),\binom{n}{3}-e\left(G\right)\right\} =\Omega\left(n^{3}\right)$,
so we must have $\left|A\right|,\left|B\right|=\Omega\left(n\right)$.
Also observe that if $x,y'\in A$ and $x',y\in B$, then
\begin{equation}
\left|\deg_{G'}\left(x\vphantom',y\right)-\deg_{G'}\left(x',y\right)-\deg_{G'}\left(x,y'\right)+\deg_{G'}\left(x',y'\right)\right|=\left|\vphantom{\deg_{G'}\left(x',y\right)}2\left(n-2\right)-\left|A\right|-\left|B\right|\right|=n-4.\label{eq:ideal-2-degrees}
\end{equation}
There are $\Omega\left(n^{4}\right)$ such choices of $\left(x,x',y,y'\right)$.
We claim that in fact there are $N'=\Omega\left(n^{4}\right)$ choices
of $\left(x,x',y,y'\right)$ such that
\begin{equation}
\deg_{G}\left(x,y\right)-\deg_{G}\left(x',y\right)-\deg_{G}\left(x,y'\right)+\deg_{G}\left(x',y'\right)\ge n/2.\label{eq:threshold-2-degrees}
\end{equation}
Indeed, recall that $G'$ is obtained from $G$ by adding and removing
$o\left(n^{3}\right)$ edges, and adding or removing an edge from
$G$ can affect the value of the above expression by at most 1, for
at most $O\left(n^2\right)$ 4-tuples $\left(x,x',y,y'\right)$. Therefore there can be only $o(n^4)$ 4-tuples which satisfy \cref{eq:ideal-2-degrees} but not \cref{eq:threshold-2-degrees}.

Now, let $H'$ be the random graph on the vertex set $\range{n/2}$
with an edge $\left\{ i,j\right\} $ if $g_{ij}\ge n/2$. We have
\[
\E e\left(H'\right)=\frac n2\left(\frac n2-2\right)\frac{N'}{n\left(n-1\right)\left(n-2\right)\left(n-3\right)}=\Theta\left(n^{2}\right),
\]
and, as in the previous case, by a concentration inequality
$e\left(H'\right)=\Omega\left(n^{2}\right)$ with probability ${1-e^{-\Omega\left(n\right)}}$,
in which case $H'$ has a matching of size $m=\Omega\left(n\right)$.
The desired result then follows from \cref{cor:second-degree-MNV},
with $d=3$ and $r=\min\left\{ m,n/8\right\} $.

\subsection{\label{subsec:3-classification}Characterising $\protect\F$-free
3-graphs}

In this subsection we prove \cref{lem:3-classification}, as an inductive
consequence of the following two lemmas. Let $K_{a,b}^{\left(3\right)}$
be the complete bipartite 3-graph with parts of sizes $a$ and $b$
(meaning that the vertex set can be partitioned into two parts of
sizes $a$ and $b$, and the edges are those triples which intersect
both parts).
\begin{lem}
\label{lem:3-classification-base-case}Under the conditions of \cref{lem:3-classification},
$G$ contains a copy of $K_{5,5}^{\left(3\right)}$ or its complement.
\end{lem}

\begin{lem}
\label{lem:3-classification-inductive-step}Consider an $\F$-free
3-graph $G$ and let $v$ be one of its vertices. Suppose that $G-v=G_{A,B,M}$
for some partition $A\cup B$ of the vertex set of $G-v$ and some
set of disjoint pairs $M\subseteq A\times B$. Suppose also that $\left|A\right|,\left|B\right|\ge5$.
Then there is a partition $A'\cup B'$ of the vertex set of $G$,
and a set of disjoint pairs $M'\subseteq A'\times B'$, satisfying
$A'\supseteq A$, $B'\supseteq B$, $M'\supseteq M$, such that $G=G_{A',B',M'}$.
\end{lem}

To prove \cref{lem:3-classification-base-case}, we use the following
theorem due to Fox and Sudakov \cite[Theorem~4.2]{FS08}, which states
that certain natural ``patterns'' are unavoidable in hypergraphs
with density bounded away from zero and one. This theorem was actually
stated in \cite{FS08} without proof (Fox and Sudakov were mainly
concerned about an analogous theorem for graphs), so for completeness
we include a proof in \cref{sec:fox-sudakov}.
\begin{thm}
\label{lem:fox-sudakov}Consider a red-blue colouring of the edges
of the complete $n$-vertex 3-graph, with $\Omega\left(n^{3}\right)$
red edges and $\Omega\left(n^{3}\right)$ blue edges, and consider
any $t\in\NN$.

If $n$ is sufficiently large then $G$ contains disjoint vertex subsets
$V_{1},V_{2},V_{3}$ each of size $t$, such that for every function
$f:\left\{ 1,2,3\right\} \to\left\{ 1,2,3\right\} $, all the edges
$\left\{ v_{1},v_{2},v_{3}\right\} $ with $v_{i}\in V_{f\left(i\right)}$
for $i\in\left\{ 1,2,3\right\} $ have the same colour, but the edge
colouring of $V_{1}\cup V_{2}\cup V_{3}$ is not monochromatic.
\end{thm}

Now we prove \cref{lem:3-classification-base-case} and \cref{lem:3-classification-inductive-step}.
Note that $n$-vertex 3-graphs are equivalent to red-blue colourings
of the complete $n$-vertex 3-graph; we switch between these points
of view interchangeably.
\begin{proof}[Proof of \cref{lem:3-classification-base-case}]
We may assume that $G$ contains disjoint vertex subsets $V_{1},V_{2},V_{3}$
satisfying the conclusion of \cref{lem:fox-sudakov}, with $t=5$.
Suppose first that for some $i_{1},i_{2}$, the set $V_{i_{1}}\cup V_{i_{2}}$
is not monochromatic and the induced colouring is not isomorphic to
$K_{5,5}^{\left(3\right)}$ or its complement. For all remaining possibilities
of the induced colouring, one can check that if $x,y,z\in V_{i_{1}}$
and $x',y',z'\in V_{i_{2}}$ then the expression in \cref{eq:3-form}
is nonzero, meaning that $G$ is not $\F$-free, which is a contradiction.
Alternatively, if each $V_{i_{1}}\cup V_{i_{2}}$ is monochromatic,
then there is only one possibility (up to swapping colours) for the
colouring of $V_{1}\cup V_{2}\cup V_{3}$, and one can check that
if $x,y,z\in V_{1}$, $x',y'\in V_{2}$ and $z'\in V_{3}$, then the
expression in \cref{eq:3-form} is nonzero, which is again a contradiction.
\end{proof}
\begin{proof}[Proof of \cref{lem:3-classification-inductive-step}]
First, consider $x_{a},x_{a}'\in A$ and $x_{b},x_{b}'\in B$ such
that $\left(x_{a},x_{b}'\right),\left(x_{a}',x_{b}\right)\notin M$.
By the assumption that $\left|B\right|\ge5$, there is $x_{b}^{*}\in B\setminus\left\{ x_{b},x_{b}'\right\} $
such that $\left(x_{a}',x_{b}^{*}\right),\left(x_{a},x_{b}^{*}\right)\notin M$.
If $\left(x,x',y,y',z,z'\right)=\left(v,x_{b}^{*},x_{a},x_{b},x_{a}',x_{b}'\right)$
then the expression in \cref{eq:3-form} is equal to

\begin{equation}
1+a_{vx_{a}x_{a}'}+a_{vx_{b}x_{b}'}-a_{vx_{a}x_{b}'}-a_{vx_{a}'x_{b}}=0.\label{eq:1}
\end{equation}
This implies that at most one of $\left\{ v,x_{a},x_{a}'\right\} $
and $\left\{ v,x_{b},x_{b}'\right\} $ can be an edge. But for any
$x_{a},x_{a}'\in A$ and any two vertices in $B$, we can assign those
two vertices the labels $x_{b}$ and $x_{b'}$ in such a way that
$\left(x_{a},x_{b}'\right),\left(x_{a}',x_{b}\right)\notin M$. So
either there is no edge containing $v$ and two vertices from $A$,
or there is no edge containing $v$ and two vertices from $B$. Without
loss of generality, suppose the former is the case.

Now, we wish to study the edges of the form $\left\{ v,x_{a},x_{b}\right\} $,
for $x_{a}\in A$ and $x_{b}\in B$. Let $\Gamma$ be the auxiliary
bipartite graph on the vertex set $A\cup B$ with an edge $\left(x_{a},x_{b}\right)\in A\times B$
if $\left\{ v,x_{a},x_{b}\right\} $ is \emph{not} an edge in $G$.
\begin{claim*}
Either $\Gamma=M$ or $\Gamma$ is obtained from $M$ by adding every
edge $\left(x_{a},x^{*}\right)$ incident to a single vertex $x^{*}\in B$
which does not appear in any pair of $M$.
\end{claim*}
\begin{proof}[Proof of claim]
For any $x_{a},x_{a}'\in A$, $x_{b},x_{b}'\in B$ with $\left(x_{a},x_{b}'\right),\left(x_{a}',x_{b}\right)\notin M$,
by \cref{eq:1} at least one of $\left\{ v,x_{a},x_{b}'\right\} $
and $\left\{ v,x_{a}',x_{b}\right\} $ is an edge. So $\Gamma\setminus M$ does not have a matching of size 2, and by K\H{o}nig's theorem the edges in
$\Gamma\setminus M$ are all incident to a single vertex $x^{*}$,
which may be in $A$ or $B$. (If $\Gamma\setminus M$ consists of
just one edge, then we can take $x^{*}\in B$, and if $\Gamma\setminus M$
has no edges, we set $x^{*}=\emptyset$).

Consider any $x_{a},x_{a}'\in A$ and $x_{b}\in B$. Choose $x_{b}^{*}\in B\setminus\left\{ x_{b}\right\} $
and $x_{a}^{*}\in A\setminus\left\{ x_{a},x_{a}'\right\} $ such
that $\left(x_{a}^{*},x_{b}\right),\left(x_{a}^{*},x_{b}^{*}\right),\left(x_{a},x_{b}^{*}\right),\left(x_{a}',x_{b}^{*}\right)\notin M$.
If $\left(x,x',y,y',z,z'\right)=\left(v,x_{b}^{*},x_{b},x_{a}^{*},x_{a},x_{a}'\right)$
then the expression in \cref{eq:3-form} is equal to
\begin{equation}
\one_{\left(x_{a}',x_{b}\right)\notin M}-\one_{\left(x_{a},x_{b}\right)\notin M}+a_{vx_{a}x_{b}}-a_{vx_{a}'x_{b}}=0.\label{eq:2}
\end{equation}
If $\left(x_{a},x_{b}\right)\in M$, then $\left(x_{a}',x_{b}\right)\notin M$, so \cref{eq:2} immediately implies that
$\left\{ v,x_{a}',x_{b}\right\} $ is an edge of $G$, and $\left\{ v,x_{a},x_{b}\right\} $
is not. This implies that $M\subseteq\Gamma$, and it also allows
us to rule out the possibility that $x^{*}$ is a vertex in $B$ that
appears in some pair of $M$ (in this case we would have proved that $\Gamma\setminus M$ has no edges incident to $x^*$, which would contradict the choice of $x^*$).

Also, if $\left(x_{a},x_{b}\right),\left(x_{a}',x_{b}\right)\notin M$
then \cref{eq:2} implies that $\left(x_{a},x_{b}\right)$ is an edge
of $\Gamma$ if and only if $\left(x_{a}',x_{b}\right)$ is an edge
of $\Gamma$. This rules out the possibility that $x^{*}\in A$, and
proves that if $x^{*}\ne\emptyset$ then actually $\left(x_{a},x^{*}\right)\in\Gamma$
for all $x_{a}$, finishing the proof of the claim
\end{proof}
Given the above claim, it now remains to consider edges of the form
$\left\{ v,x_{b},x_{b}'\right\} $, for $x_{b},x_{b}'\in B$. For
any $x_{b},x_{b}'\in B\setminus\left\{ x^{*}\right\} $ we may choose
$x_{a},x_{a}'\in A$ such that $\left(x_{a},x_{b}'\right),\left(x_{a}',x_{b}\right)\notin M$,
so by \cref{eq:1}, $\left\{ v,x_{b},x_{b}'\right\} $ is an edge of
$G$. If $\Gamma=M$ we may now conclude that $G=G_{A',B,M}$ for
$A'=A\cup\left\{ v\right\} $.

Alternatively, if $\Gamma\ne M$, we may similarly deduce from \cref{eq:1}
that for any $x_{b}\in B\setminus\left\{ x^{*}\right\} $, $\left\{ v,x_{b},x^{*}\right\} $
is not an edge of $G$. We may then conclude that $G=G_{A',B,M'}$
for $A'=A\cup\left\{ v\right\} $ and $M'=M\cup\left\{ \left(v,x^{*}\right)\right\} $.
\end{proof}

\section{\label{sec:concluding}Further directions of research}

\subsection{Possible generalisations and improvements}

An obvious conjecture is that the logarithmic term in \cref{thm:almost}
is unnecessary, as follows.
\begin{conjecture}
\label{conj:kl}For all $k$ and all $0\le\ell\le\binom{k}{2}$, let
$\minell=\min\left\{ \ell,\binom{k}{2}-\ell\right\} $. We have
\[
\ind\left(k,\ell\right)=O\left(\sqrt{k/\minell}\right).
\]
\end{conjecture}

We remark that the logarithmic term in \cref{thm:almost} arises purely
because of the corresponding logarithmic term in \cref{thm:MNV}. In
turn, this logarithmic term arises from an estimate due to Kane for
the so-called Gotsman--Linial conjecture, as follows. For
an $n$-variable Boolean function $f:\left\{ -1,1\right\} ^{n}\to\RR$,
the \emph{average sensitivity} of $f$
is defined by
\[
\AS\left(f\right)=\sum_{i=1}^{n}\Pr\left(f\left(\gamma_{1},\dots,\gamma_{i-1},\gamma_{i},\gamma_{i+1},\dots,\gamma_{n}\right)\ne f\left(\gamma_{1},\dots,\gamma_{i-1},-\gamma_{i},\gamma_{i+1},\dots,\gamma_{n}\right)\right),
\]
where $\g\in\Rad^{n}$. The Gotsman--Linial conjecture essentially\footnote{The original conjecture due to Gotsman and Linial \cite{GL94} was
slightly sharper than what is stated here, but it was recently disproved
\cite{Cha18,KMW}.} states that if a Boolean function $f$ has the form $f\left(\boldsymbol{x}\right)=\one_{p\left(\boldsymbol{x}\right)>0}$
for a degree-$d$ polynomial $p$ (that is, $f$ is a degree-$d$
\emph{polynomial threshold function}), then $\AS\left(f\right)=O\left(d\sqrt{n}\right)$.
Kane \cite{Kan14} proved that if $f$ is a degree-$O(1)$ threshold
function then $\AS\left(f\right)=\sqrt{n}\log^{O\left(1\right)}n$.
To prove \cref{conj:kl} via the methods in \cite{MNV16} and the methods
in \cref{sec:almost}, it would suffice to prove that $\AS\left(f\right)=O\left(\sqrt{n}\right)$
under the same assumptions. In particular, a bound of the form $\AS\left(f\right)\le g\left(n\right)$
would imply that under the conditions of \cref{conj:kl} we have $\ind\left(k,\ell\right)=O\left(g\left(\minell/k\right)k/\minell\right).$

Next, the appropriate generalisation of \cref{conj:kl} to hypergraphs
seems to be as follows.
\begin{conjecture}
\label{conj:hypergraphs}For any $r,k$ and any $0\le\ell\le\binom{k}{r}$,
let $\minell=\min\left\{ \ell,\binom{k}{r}-\ell\right\} $. We have
\[
\ind_{r}\left(k,\ell\right)=O\left(\sqrt{k^{r-1}/\minell}\right).
\]
\end{conjecture}

It is likely that the arguments in \cref{sec:hypergraph-sqrt} can
be pushed to hypergraphs with uniformity higher than three, but our
proof of \cref{thm:hypergraph-sqrt} involves some rather complicated
checking of cases, which would likely be even more complicated for
higher uniformities. To be specific, the part of the proof that really
depends on the uniformity-3 assumption is \cref{lem:3-classification},
which (essentially) classifies the 3-graphs such that
\[
a_{xyz}-a_{xyz'}-a_{xy'z}-a_{x'yz}+a_{xy'z'}+a_{x'yz'}+a_{x'y'z}-a_{x'y'z'}=0
\]
for every choice of distinct vertices $\left(x,x',y,y',z,z'\right)$.
In general, one might try to classify the $r$-graphs
such that

\[
\sum_{\bb\in\left\{ 0,1\right\} ^{r}}\left(-1\right)^{\left|\bb\right|}a_{z_{1}^{\b_{1}}\dots z_{r}^{\b_{r}}}=0
\]
for every choice of distinct vertices $\left(z_{1}^{0},z_{1}^{1},z_{2}^{0},z_{2}^{1},\dots,z_{r}^{0},z_{r}^{1}\right)$.
As a natural family of $r$-graphs with this property,
consider the complete $\left(r-1\right)$-partite $r$-graphs
$K_{n_{1},\dots,n_{r-1}}^{\left(r\right)}$, whose vertex set is partitioned
into $r-1$ parts of sizes $n_{1},\dots,n_{r-1}$, and whose edge
set is defined to consist of all sets of $r$ vertices which intersect
every part. Is it true that every $r$-graph with the
aforementioned property resembles some $K_{n_{1},\dots,n_{r-1}}^{\left(r\right)}$?
If so, using the same approach as for \cref{thm:hypergraph-sqrt} one
could prove that if $\minell=\Omega_{k}\left(k^{r}\right)$ then 
\[
\ind_{r}\left(k,\ell\right)\le\frac{\log^{O\left(1\right)}k}{\sqrt{k}}.
\]
However, since our proof uses the induced hypergraph removal lemma,
new ideas would be necessary to address the ``sparse'' case where
$\minell=o_{k}\left(k^{r}\right)$.

As observed in \cite{AHKT}, it also makes sense to generalise \cref{conj:1/e}
to hypergraphs:
\begin{conjecture}
For all $0<\ell<\binom{k}{r}$ we have $\ind_{r}\left(k,\ell\right)\le1/e+o_{k}\left(1\right)$.
\end{conjecture}

It would be interesting even to prove a bound of the form $\ind_{r}\left(k,\ell\right)\le1-\varepsilon+o_{k}\left(1\right)$
for any $\varepsilon>0$ that does not depend on $r$; perhaps this
can be accomplished with a careful bare-hands fourth-moment computation
as in \cite{AHKT}. We remark however that the bound $\ind_{r}\left(k,\ell\right)\le1-2^{-4/3}3^{-16r}$
in \cref{thm:1-eps} is clearly not best-possible; instead of using
hypercontractivity for $\hyp\left(n,n/2\right)$ via \cref{lem:hyp-anticoncentration},
it would have been possible to give a less direct proof using the
much more developed theory of hypercontractivity for $\Rad^{n}$,
then appealing to the invariance principle in \cite{FM16}. This would
have given the bound $\ind_{r}\left(k,\ell\right)\le1-9^{-r}+o_{k}\left(1\right)$.
Of course, this still depends exponentially on $r$.

\subsection{Ramsey graphs}

It is also interesting to study the probabilities $\Pr\left(X_{G,k}=\ell\right)$
for restricted classes of (hyper)graphs $G$. If these probabilities
are small it would seem to give some evidence that the graphs in question
are very ``diverse'' or ``disordered''. In particular, say that
a graph is \emph{$C$-Ramsey} if it has no clique or independent set
of size $C\log_{2}n$. There has been a lot of work on diversity of
Ramsey graphs from various points of view; in particular, Kwan and
Sudakov \cite{KS2} recently resolved a conjecture of Erd\H os, Faudree
and S\'os which effectively says that if $G$ is an $O\left(1\right)$-Ramsey
graph then for many values of $k$, the random variables $X_{G,k}$
have large support (see also \cite{AB89,AB07,AK09,ABKS09,NST16,KS1} for related work). It would be very interesting to study the probabilities
$\Pr\left(X_{G,k}=\ell\right)$ for Ramsey graphs. For example, if
$G$ is an $n$-vertex $O\left(1\right)$-Ramsey graph, is it true
that
\[
\Pr\left(X_{G,n/2}=\ell\right)=O\left(1/n\right)
\]
for all $\ell$?

As mentioned in \cite{KS1}, it would also be interesting to ask a
more tractable version of this question for $X=e\left(G\left[A\right]\right)$,
where $A$ is a uniformly random vertex subset of an $O\left(1\right)$-Ramsey
graph $G$. In this case $X$ can be interpreted as a quadratic polynomial
of a random vector $\g\in\Rad^{n}$, so this question is closely related
to a conjecture of Costello \cite[Conjecture~3]{Cos13} attempting
to characterise the quadratic polynomials $f$ in $n$ variables with point probabilities
$\Pr\left(f\left(\g\right)=\ell\right)$ much larger than $1/n$ (see
also a related inverse theorem of Nguyen \cite{Ngu12}).

\subsection{Anticoncentration ``on the slice''}

Historically, almost all the work on anticoncentration has focused
on sums or low-degree polynomials of independent random variables
(for example, see the survey of Nguyen and Vu \cite{NV13} concerning
the Littlewood--Offord problem and its variants). A recent
exception is a Littlewood--Offord-type theorem ``on the
slice'' due to Litvak, Lytova, Tikhomirov, Tomczak-Jaegermann and
Youssef \cite[Proposition~4.10]{LLTTY17}. The aforementioned authors
proved an upper bound on the point probabilities $\Pr\left(f\left(\x\right)=\ell\right)$
for an $n$-variable degree-1 polynomial $f$ of a random vector $\x\in\hyp\left(n,n/2\right)$,
using the coupling in \cref{fact:coupling}, and used this result to
study the singularity probability of a random zero-one matrix with
fixed row and column sums. Specifically, they were able to show that
$f\left(\x\right)$ has good anticoncentration if $f$ has many pairs
of different coefficients, which means that $f$ is far from a multiple
of the polynomial $x_{1}+\dots+x_{n}$.

In this paper we generalised the above methods to higher-degree polynomials,
and effectively showed that for a degree-$d$ polynomial $f$, the
random variable $f\left(\x\right)$ has good anticoncentration if
there are many $2d$-tuples of coefficients which satisfy a
certain inequality. This is in some sense a combinatorial criterion,
and it would be interesting if an algebraic criterion could be proved
to also suffice. For example, does $f\left(\x\right)$ have good anticoncentration
whenever $f$ is in some sense far from a polynomial with $\left(x_{1}+\dots+x_{n}\right)$
as a factor?

There is also the possibility that more natural anticoncentration
theorems could be stated in terms of harmonic polynomials (recall
the definition from \cref{subsec:hypercontractivity}), which are in
some sense the correct representation of functions on the slice (see
for example \cite{FM16}). For harmonic polynomials of $\hyp\left(n,pn\right)$
random vectors, the invariance principle in \cite{FM16} can be used
to apply standard anticoncentration theorems, but the error terms
in this invariance principle prevent one from obtaining optimal bounds
in this way.

Finally, one could also study anticoncentration phenomena for more
general combinatorial random variables; for example, functions of
random permutations, as in Hoeffding's combinatorial central limit
theorem \cite{Hoe51}.

\vspace{0.3cm}
\noindent
{\bf Acknowledgments.}\, The authors would like to thank the referee and Lisa Sauermann for their careful reading of the manuscript and their valuable comments.


\begin{appendices}
\crefalias{section}{appsec}

\section{\label{sec:fox-sudakov}Unavoidable patterns in Hypergraphs}
For the sake of completeness and the convenience of the reader, in this section we provide a proof of (a generalisation of) \cref{lem:fox-sudakov}. This theorem appears as \cite[Theorem~4.2]{FS08}, but was stated without proof.
\begin{thm}
	\label{thm:fox-sudakov}
	For each $\varepsilon>0$ and positive integers $r$ and $t$, there is a positive integer $N=N(t,r,\varepsilon)$ such that the following holds. Consider any red-blue colouring
	of the complete $r$-graph with $n\ge N$ vertices which has at least $\varepsilon n^r$ edges in each colour. Then there are disjoint vertex subsets $V_{1},\dots,V_{r}$, each of size $t$, such that for every function
	$f:\left\{ 1,\dots,r\right\} \to\left\{ 1,\dots,r\right\} $, all
	the edges $\left\{ v_{1},\dots,v_{r}\right\} $ with $v_{i}\in V_{f\left(i\right)}$
	for $i\in\left\{ 1,\dots,r\right\} $ have the same colour, but the
	edge colouring of $V_{1}\cup\dots\cup V_{r}$ is not monochromatic.
\end{thm}

We remark that our proof will actually give the slightly stronger statement that the edges contained in $V_1$ have a different colour than the edges with a vertex in each $V_i$.

At the heart of our proof of \cref{thm:fox-sudakov} is the following lemma.

\begin{lem}
	\label{lem:mixed-complete}
	For each $\varepsilon>0$ and positive integers $r$ and $q$, there is a positive integer $M=M(q,r,\varepsilon)$ such that the following holds. Consider any red-blue colouring
	of the complete $r$-graph with $n\ge M$ vertices which has at least $\varepsilon n^r$ edges in each colour. Then there exist disjoint vertex subsets
	$V_{1},\dots,V_{r-1},R,B$, each of size $q$, such that
	\begin{compactenum}	
		\item[\rm (i)] $\{v_1,\ldots,v_{r-1},v\}$ is coloured red for $v_1 \in V_1,\ldots,v_{r-1}\in V_{r-1},v\in R$, and
		\item[\rm (ii)] $\{v_1,\ldots,v_{r-1},v\}$ is coloured blue whenever $v_1 \in V_1,\ldots,v_{r-1}\in V_{r-1},v\in B$.
	\end{compactenum}
\end{lem}

We will deduce \cref{lem:mixed-complete} from the following lemma together with some classical extremal results.
	
\begin{lem}
\label{claim:high-mixed-degree}
For any $r\in\NN$ and $\varepsilon>0$,
let $\alpha_{r}\left(\varepsilon\right)=\left(\varepsilon/3\right)^{4^{r}}$.
Consider a red-blue colouring of $\binom{\range{n}}{r}$ with at least $\varepsilon n^{r}$ edges in each colour. Then there are $\alpha_{r}\left(\varepsilon\right)n^{r-1}$ $\left(r-1\right)$-sets of vertices that are simultaneously contained
in $\alpha_{r}\left(\varepsilon\right)n$ red edges and $\alpha_{r}\left(\varepsilon\right)n$
blue edges.
\end{lem}

\begin{proof}
We will prove by induction the stronger version of this statement
for \emph{partial} red-blue colourings where we allow $\alpha_{r+1}\left(\varepsilon\right)n^{r}$
edges to be uncoloured. The base case $r=1$ is trivial, so assume
that this claim holds for all uniformities less than some $r\ge2$.

Suppose to the contrary that this statement is false. Let $\R$ contain the $(r-1)$-sets of vertices which are in fewer than $\alpha_{r}\left(\varepsilon\right)n$
blue edges, and let $\B$ contain the $(r-1)$-sets which are in fewer than $\alpha_{r}\left(\varepsilon\right)n$
red edges. Then let $\S$ contain the remaining $(r-1)$-sets, which are each simultaneously contained in $\alpha_{r}\left(\varepsilon\right)n$ red edges and $\alpha_{r}\left(\varepsilon\right)n$
blue edges. We are assuming that $\left|\S\right|<\alpha_{r}\left(\varepsilon\right)n^{r-1}$.
The number
of red edges is then at most $\left|\R\right|n+\left|\S\right|n+\alpha_{r}\left(\varepsilon\right)n\left|\B\right|\le\left|\R\right|n+\alpha_{r}\left(\varepsilon\right)n^{r}+\alpha_{r}\left(\varepsilon\right)n^{r}$,
so $\left|\R\right|\ge\left(\varepsilon-\alpha_{r}\left(\varepsilon\right)-\alpha_{r}\left(\varepsilon\right)\right)n^{r-1}\ge\left(\varepsilon/3\right)n^{r-1}$.
We may similarly deduce that $\left|\B\right|\ge\left(\varepsilon/3\right)n^{r-1}$.

Now, consider the red-blue colouring of the edges of the complete
$n$-vertex $\left(r-1\right)$-graph, where we colour an edge red
if it is in $\R$ and blue if it is in $\B$. By induction, since at most $|\S|<\alpha_{r}\left(\varepsilon\right)n^{r-1}$ edges are uncoloured, there are
$\alpha_{r-1}\left(\varepsilon/3\right)n^{r-2}$ $\left(r-2\right)$-sets
that are simultaneously contained in $\alpha_{r-1}\left(\varepsilon/3\right)n$
sets from $\R$ and $\alpha_{r-1}\left(\varepsilon/3\right)n$
sets from $\B$. Since fewer than $\alpha_{r+1}\left(\varepsilon\right)n^{r}$
edges are uncoloured in total, we can find such an $(r-2)$-set $Z$ with the extra property that $Z$ is contained in at most 
$$\frac{\binom{r}{2}\alpha_{r+1}(\varepsilon)n^r}{\alpha_{r-1}(\varepsilon/3)n^{r-2}}< \alpha_r(\varepsilon)n^2$$
uncoloured edges. Let
$\Q$ be the collection of all $r$-sets of the form $X\cup Y$, for
$Z\subseteq X\in \R$ and $Z\subseteq Y\in \B$. Note that $X,Y$ are uniquely determined by their union because we can write $(X\cup Y)\setminus Z=\{x,y\}$, where $\{x\}\cup Z=X$ is red and $\{y\}\cup Z=Y$ is blue. So $\left|\Q\right|\ge\left(\alpha_{r-1}\left(\varepsilon/3\right)n\right)^{2}$.
On the other hand, let $N_Z\le n$ be the number of $X\in \R$ which include $Z$. By the choice of $\R$, fewer than $N_Z\alpha_{r}\left(\varepsilon\right)n\le\alpha_{r}\left(\varepsilon\right)n^{2}$
of elements of $\Q$ are blue, and by the choice of $\B$ fewer than $\alpha_{r}\left(\varepsilon\right)n^{2}$
are red. But we have seen above that fewer than $\alpha_{r}\left(\varepsilon\right)n^{2}$ are uncoloured, so $\left|\Q\right|\le 3\alpha_{r}\left(\varepsilon\right)n^2$. One can check that 
\[
\alpha_{r-1}\left(\varepsilon/3\right)^{2}> 3\alpha_{r}\left(\varepsilon\right),
\]
yielding our desired contradiction.
\end{proof}

\begin{proof}[Proof of \cref{lem:mixed-complete}]
	By \cref{claim:high-mixed-degree}, there is a collection $\S$ of $\Omega(n^{r-1})$ $(r-1)$-sets of vertices that are simultaneously contained in $\Omega(n)$ red edges and $\Omega(n)$ blue edges. Let $G_{\mathrm{red}}$ be the bipartite graph with vertex set $V(G_{\mathrm{red}})=\S\cup \range{n}$ and edge set $E(G_{\mathrm{red}})=\left\{(S,v):S\in \S,v\in \range{n},S\cup \{v\} \enskip \text{is red}\right\}$, and define $G_{\mathrm{blue}}$ in exactly the same way, using blue edges instead of red edges. As $e(G_{\mathrm{red}})=\Omega(|\S|n)=\Omega(n^r)$, it follows from the K\H{o}v\'ari--S\'os--Tur\'an theorem \cite{KST} that $G_{\mathrm{red}}$ must contain a complete bipartite graph with parts $\S'\subseteq \S$ and $R\subseteq \range{n}$ satisfying $|\S'|=|\S|^{1-o(1)}=n^{r-1-o(1)}$ and $|R|=q$. Similarly, applying the K\H{o}v\'ari--S\'os--Tur\'an theorem to the induced subgraph $G_{\mathrm{blue}}[\S'\cup \range n]$, we can find a complete bipartite subgraph of $G_{\mathrm{blue}}$ with parts $\S''\subset \S'$ and $B\subset \range{n}$ such that $|\S''|=|\S'|^{1-o(1)}=n^{r-1-o(1)}$ and $|B|=q$. Since $\S''$ is an $(r-1)$-graph on $\range{n}$ with $n^{r-1-o(1)}$ edges, a result due to Erd\H{o}s \cite[Theorem 1]{Erdos64} (essentially generalising the K\H{o}v\'ari--S\'os--Tur\'an theorem to hypergraphs) tells us that $\S''$ contains a complete $(r-1)$-partite $(r-1)$-graph whose parts $V_1,\ldots,V_{r-1}$ have the same size $q$. Clearly, the vertex subsets $V_1,\ldots,V_{r-1},R,B$ have the desired properties.
\end{proof}

To prove \cref{thm:fox-sudakov} we also need the following Ramsey-type result.

\begin{lem}
	\label{lem:refine}
	For all $r,t\in\NN$, there is $Q_{r}\left(t\right)\in \NN$
	such that the following holds. Consider a red-blue colouring of the
	edges of the complete $r$-graph, and consider vertex sets $V_{1}',\dots,V_{r}'$
	each of size at least $Q_r\left(t\right)$, such that all the edges
	with a vertex in each $V_{i}'$ have the same colour, and all the
	edges within $V_{1}'$ have the other colour. Then there are subsets
	$V_{i}\subseteq V_{i}'$ of size $t$ satisfying the conclusion of
	\cref{thm:fox-sudakov}.
\end{lem}

\begin{proof}
	For every function $f:\left\{ 1,\dots,r\right\} \to\left\{ 1,\dots,r\right\} $,
	we say that $V_{1},\dots,V_{r}$ is \emph{$f$-good} if all the edges
	$\left\{ v_{1},\dots,v_{r}\right\} $ with $v_{i}\in V_{f\left(i\right)}$
	for $i\in\left\{ 1,\dots,r\right\} $ have the same colour. Provided $Q_r(t)$ is large enough, we can
	iteratively apply 
	the Product Ramsey Theorem (see for example \cite[Theorem~9.2]{Pro13}) to shrink the $V_{i}'$
	until they are $f$-good for every $f$.
\end{proof}

We are finally ready to prove \cref{thm:fox-sudakov}.
\begin{proof}[Proof of \cref{thm:fox-sudakov}]
	By Ramsey's theorem, there is a function $R_{r}:\NN \rightarrow \NN$
	such that every red-blue colouring of the edges of the complete 
	$r$-graph on $R_{r}\left(k\right)$ vertices has a monochromatic $k$-clique.
	
	We apply \cref{lem:mixed-complete} to obtain vertex subsets $V_1'',V_2''\ldots,V_{r-1}'',R,B$ each of size $R_r(Q_r(t))$.
	Let $V_{1}'$ be a monochromatic $Q_{r}\left(t\right)$-clique in $V_{1}''$, and
	assume without loss of generality that it is red. Choose $V_2'\subset V_2'',\ldots,V_{r-1}'\subset V_{r-1}''$ and $V_r'\subset B$ such that $|V_i'|=Q_r(t)$ for every $2\le i \le r$.
	Then apply \cref{lem:refine}.
\end{proof}

\end{appendices}
\end{document}